\documentclass[11pt,reqno]{amsart}
\usepackage{amsfonts,amssymb,amsmath,amsthm}
\usepackage[margin=1in]{geometry}
\usepackage[hidelinks]{hyperref}
\usepackage{cleveref}
\usepackage{enumitem}
\usepackage{microtype}
\usepackage{mathtools}
\usepackage{color}
\usepackage{tikz-cd}
\usepackage{comment}

\def\MR#1{}

\def\Z{\mathbb{Z}}
\def\Q{\mathbb{Q}}

\def\C{\mathbb{C}}
\def\F{\mathbb{F}}
\def\Gal{\operatorname{Gal}}
\def\GL{\operatorname{GL}}
\def\SL{\operatorname{SL}}

\def\GSp{\operatorname{GSp}}
\def\det{\operatorname{det}}
\def\tr{\operatorname{tr}}

\def\Im{\operatorname{Im}}
\def\Frob{\operatorname{Frob}}
\def\Aut{\operatorname{Aut}}

\newcommand\sm[1]{\begin{psmallmatrix}#1\end{psmallmatrix}}
\newcommand\m[1]{\begin{pmatrix}#1\end{pmatrix}}

\def\ov#1{\overline{#1}}
\def\wt#1{\widetilde{#1}}
\def\cl#1{\mathcal{#1}}

\def\thin{{\hskip 1pt}}
\def\fr#1{\mathfrak{#1}}
\def\ol#1{\overline{#1}}
\def\Stab{\operatorname{Stab}}
\def\End{\operatorname{End}}
\def\id{\operatorname{id}}

\def\N{\mathrm{N}}
\def\ord{\mathrm{ord}}
\def\Ker{\mathrm{Ker}}
\newcommand{\simarrow}{\xrightarrow{\;\sim\;}}

\theoremstyle{plain}
\newtheorem{theorem}{Theorem}[section]
\newtheorem{corollary}[theorem]{Corollary}
\newtheorem{lemma}[theorem]{Lemma}
\newtheorem{proposition}[theorem]{Proposition}

\theoremstyle{definition}
\newtheorem{definition}[theorem]{Definition}

\theoremstyle{remark}
\newtheorem{remark}[theorem]{Remark}

\title[Ordinary primes for $\GL_2$-type abelian varieties and weight $2$ modular forms]{Ordinary primes for $\GL_2$-type abelian varieties and \\ weight $2$ modular forms}
\author[Wang]{Tian Wang}
\address{Department of Mathematics \&
Statistic, Concordia University, Montreal, Quebec, Canada}
\email{tian.wang@concordia.ca}
\author[Zhang]{Pengcheng Zhang}
\address{Max Planck Institute for Mathematics, Vivatsgasse 7, 53111 Bonn, Germany}
\email{pzhang@mpim-bonn.mpg.de}
\date{\today}

\begin{document}

\begin{abstract}
Let $A$ be a $g$-dimensional abelian variety defined over a number field $F$. It is conjectured that the set of ordinary primes of $A$ over $F$ has positive density, and this is known to be true when $g=1, 2$,  or for certain abelian varieties with extra endomorphisms. In this paper, we extend the family of abelian varieties 
whose sets of ordinary primes have positive density. Specifically, we show that if the endomorphism algebra of $A$ contains a number field $K$ of degree $g$, then under certain conditions on the fields $F$ and $K$, the set of ordinary primes of $A$ over $F$ has positive density. This includes  $\GL_2$-type abelian varieties over $\Q$ (resp.~quadratic number fields) of dimension $q$ or $2q$ (resp.~$q$) for any rational prime $q$. The proof is carried out in the general setting of compatible systems of Galois representations, and as a consequence, it also implies a positive density result for the sets of ordinary primes of certain modular forms of weight $2$. 
\end{abstract}

\maketitle

\section{Introduction}

Let $A$ be an abelian variety of dimension $g$ over a number field $F$. Let $v$ be a prime of $F$ that is of good reduction for $A$ and let $A_v$ denote the reduction of $A$ at $v$. Then, $A$ is said to \emph{have ordinary reduction at $v$} if $A_v$ is an ordinary abelian variety over $\F_v$. Equivalently, if we let $p$ be the rational prime lying below $v$ and let $P_{A,v}(X)$ denote the characteristic polynomial of the Frobenius endomorphism of $A_v$, then $A$ has ordinary reduction at $v$ if the coefficient $a_{v,g}(A)$ of $X^g$ in $P_{A,v}(X)$ is not divisible by $p$, i.e., $p\nmid a_{v,g}(A)$. We will use the phrases ``$A$ has ordinary reduction at $v$'', ``$A$ is ordinary at $v$'', and ``$v$ is an ordinary prime of $A$ over $F$'' interchangeably. In the case when $A/\Q$ is an elliptic curve, a rational prime $v=p$ is an ordinary prime of $A$ if and only if $p\nmid a_{p, 1}(A)$, where $a_{p, 1}(A)=p+1-|A_p(\F_p)|$ is the Frobenius trace of $A$ at $p$. When $p\geq 5$, this is also equivalent to that $a_{p,1}(A)\neq 0$ by the Hasse--Weil bound. 

The study of ordinary primes for abelian varieties plays a key role in arithmetic geometry.  For example, ordinary abelian varieties over $\ol{\F_v}$ are generic in the moduli space of all principally polarized abelian varieties  \cite{NoOo1980}, they exhibit canonical lifting properties to characteristic 0 via Serre--Tate theory \cite{Ka1981}, and they are crucial in the proof of Hecke orbit conjecture for certain Shimura varieties \cite{Chai2005, Chai2006}.

A conjecture commonly attributed to Serre  \cite[\#133]{Se2013} predicts that there are infinitely many ordinary primes for $A/F$. Furthermore, it is conjectured that the density of ordinary primes for $A$ over $\ol{F}$ is 1 (see e.g., \cite[\S 7]{Pink1998}). The conjecture is known when the dimension of $A$ is 1~or~2 \cite{Se2013, Ogus1982, Sawin2016}. For dimension greater than $2$, the known cases typically assume that $A$ has extra endomorphisms. Our results are also in this direction: we prove that this conjecture holds for certain higher dimensional ``$\GL_2$-type'' abelian varieties. In \Cref{result-comparison}, we will discuss other known cases of the conjecture in higher dimensions as well as draw comparison to our results.

The main results we have are the following. All the results concern ``$\GL_2$-type'' abelian varieties: the first two are for varieties of specific dimensions over $\Q$ and quadratic number fields, respectively, and the last one is for varieties over general number fields with no restriction on the dimension but subject to certain conditions on the splitting behaviors of primes. Throughout, we denote by $\End_F(A)$ the ring of $F$-endomorphisms of an abelian variety $A/F$.

\begin{theorem}
\label{ablian-var-theorem-2}
Let $A/\Q$ be an abelian variety of dimension $q$ (resp.~$2q$) for some rational prime~$q$.  
Assume that $\End_{\Q}(A)\otimes\Q$ contains a number field $K$ of degree $q$ (resp.~$2q$). Then, the set of ordinary primes of $A$ over $\Q$ has positive density. 
\end{theorem}

\begin{theorem}
\label{abelian-var-quadratic-theorem}
Let $A/F$ be an abelian variety of dimension $q$ over a quadratic number field $F$  for some rational prime $q$. Assume that $\End_F(A)\otimes\Q$ contains a number field $K$ of degree $q$. Then, the set of ordinary primes of $A$ over $F$ has positive density.
\end{theorem}

\begin{theorem}
\label{ablian-var-theorem-3}
Let $A/F$ be an abelian variety of dimension $g$ over a number field $F$. Assume that
\begin{enumerate}[label=\rm{(\arabic*)}]
    \item $A$ is absolutely simple, i.e., $A/\ol{F}$ is simple;
    \item $\End_F(A)\otimes \Q$ contains a number field $K$ of degree $g$; 
    \item $A/\ol{F}$ is not of type IV in Albert's classification.
\end{enumerate}
Suppose further that there exists a rational prime $p$ that lies below a degree $1$ prime of $F$ and either
\begin{enumerate}[label=\rm{(P\arabic*)}]
    \item remains inert in $K$, or
    \item splits into two distinct primes in $K$ of the same inertia degree.
\end{enumerate} 
Then, the set of ordinary primes of $A$ over $F$ has positive density. 
\end{theorem}

\begin{remark}
\label{result-comparison}
We make the following remarks on \Cref{ablian-var-theorem-2}, \Cref{abelian-var-quadratic-theorem}, and \Cref{ablian-var-theorem-3}.
\begin{enumerate}[label=(\arabic*)]
    \item The ``$q$'' case of \Cref{ablian-var-theorem-2}, \Cref{abelian-var-quadratic-theorem}, and the (P1) case of \Cref{ablian-var-theorem-3} are in a way immediate consequences of simple analysis of the traces of Frobenius. However, one needs a more detailed analysis of the Galois representations attached to $A$ in order to prove the ``$2q$'' case of \Cref{ablian-var-theorem-2} and the (P2) case of \Cref{ablian-var-theorem-3}. In these cases, we first prove a result on the surjectivity of the Galois representations attached to $A$ (see \Cref{Galois-images}) and then use it to deduce the positive density result.
    \item In \cite{Fi2024}, Fit\'{e} showed that the sets of ordinary primes of certain abelian varieties with extra endomorphisms have positive density. In particular, our results can be viewed as certain generalizations of \cite[Theorem~3]{Fi2024} from dimension $3$ to certain other dimensions, though the techniques we use are quite different from the ones by Fit\'{e}.
    \item In \cite{Su2020}, Suh proved some results on the positive density of the sets of ordinary primes for certain Hilbert modular forms of parallel weight $2$.\footnote{Hilbert modular forms of parallel weight $2$ conjecturally correspond to certain abelian varieties of $\GL_2$-type via some generalization of the Eichler--Shimura theory.} One should compare the ``$2q$'' case of \Cref{ablian-var-theorem-2} and the (P2) case of \Cref{ablian-var-theorem-3} with the conditional result \cite[Theorem.~(d), p.649]{Su2020}. Our proof in these cases ultimately reduces to showing that $\{v:\N_{K/\Q}(a_v)^2\cdot\N(v)^{-g}\in\Z\}$ has density $0$, where $a_v$ is the trace of Frobenius at $v$ of some $2$-dimensional Galois representation attached to $A$. This should follow if one assumes some strong conjectural variant of the Sato--Tate conjecture as considered in \cite[Section~3.3]{Su2020}. In our case, we prove this unconditionally by a careful analysis of the image of the Galois representation attached to $A$.
    \item The essential parts of our proofs depend only on the properties of the traces of Frobenius and the Galois representations attached to these abelian varieties. 
    We thus formulate those parts separately as \Cref{galois-rep-theorem} and \Cref{inert-in-K-theorem}. Indeed, one immediate benefit is that one can apply these to obtain similar results on modular forms and Hilbert modular forms (see \Cref{modular-form-theorem} and \Cref{galois-rep-theorem-remark} (2)).
    \item Other known cases of the positive density conjecture on the ordinary primes include the case when $A/F$ is a CM abelian variety \cite[Remark 12]{Fi2024} (i.e., the geometric endomorphism algebra $\End_{\ol{F}}(A)\otimes \Q$ is a CM field of degree $2g$) or when the Mumford--Tate group of $A$ is small in the sense of \cite[Theorem 2.2, p. 168]{Noot1995} or \cite[Theorem 7.1]{Pink1998}. More recently,  Cantoral Farf\'an, Li, Mantovan, Pries, and Tang \cite{CLMPT2025} established a positive density result for the sets of ordinary primes of certain abelian varieties parametrized by unitary Shimura curves with simple signature.
\end{enumerate}
\end{remark}

\begin{remark}
We provide an example  of  a family of hyperelliptic curves whose Jacobians satisfy the condition (P1) \text{or (P2).} 
Consequently, the sets of ordinary primes of these abelian varieties have positive density.

The family we consider is taken from \cite{TaToVe1991}. Let $p$ be a rational prime with $p\neq 2$ or $5$.
Let $\zeta_p$ denote the primitive $p$-th root of unity in $\ol{\Q}$, $g(X)\in \Z[X]$ be the minimal polynomial of $-(\zeta_p+\zeta_p^{-1})$, and   $f_t(X):= Xg(X^2-2)+t\in \ol{\Q}[X]$ for any $t\in \ol{\Q}$. The equation 
\[
C_t: y^2=f_t(x)
\]
characterizes a family of absolutely simple hyperelliptic curves of genus $\frac{p-1}{2}$ \mbox{\cite[Corollary 6]{TaToVe1991}}. For each curve $C_t$, the Jacobian variety  $J_{C_t}$ is defined over $F:=\Q(t)$, of dimension $\frac{p-1}{2}$, and  contains the totally real field $K:=\Q(\zeta_p+\zeta_p^{-1})$ in its  endomorphism algebra \cite[Theorem 1]{TaToVe1991}. 
Note that $[K:\Q]=\frac{p-1}{2}=\dim J_{C_t}$.  Since the Galois extension $K/\Q$ is cyclic when $p\geq 3$, if $\ell$ is inert in $K$, then the order of $\ell \pmod p$ is divisible by $\frac{p-1}{2}$, and if $\ell$ splits into two primes of the same inertia degree, then the order of $\ell\pmod p$ is  divisible by $\frac{p-1}{4}$ (if $p\equiv 1\pmod{4}$). Clearly, such a prime $\ell$ exists.   Thus, the condition (P1) is satisfied if $p\geq 3$, and the condition (P2) is satisfied if $p\equiv 1\pmod{4}$.
Finally, by choosing a number field $F$  such that $\ell$ lies below a degree $1$ prime of $F$ and taking $t$ as the primitive element of $F$, we see that $\{J_{C_t}\}_{t\in F}$ is a family of abelian varieties whose sets of ordinary primes over $F$ have positive density. 
\end{remark}

A comparable notion of ``ordinary" also appears in the study of modular forms. Let $f$ be a newform of weight $k\geq 2$ and level $N$, let $a_n(f)$ be the $n$-th Fourier coefficient of $f$, and let $K_f$ be the coefficient field of $f$. For a rational prime $p\nmid N$ and a prime $\lambda|p$ of $K_f$, $f$ is called \emph{$\lambda$-ordinary} if $\lambda \nmid (a_p(f))$, and $f$ is called \emph{$p$-ordinary} if $f$ is $\lambda$-ordinary for all $\lambda|p$, or equivalently, $\lambda \nmid(a_p(f))$ for all primes $\lambda |p$ of $K_f$. We will use the phrases ``$f$ is $p$-ordinary'', ``$f$ is ordinary at $p$'', and ``$p$ is an ordinary prime of $f$ over $\Q$" interchangeably.

The notion of ordinary modular forms first appeared in a paper by Hida \cite{Hi1981-ord}. This paper was followed by a series of works by Hida, where he established the theory of ordinary $p$-adic modular forms, or nowadays known as Hida theory.  The ordinary assumption has since played an important role in the theory of ($p$-adic) modular forms, particularly in Iwasawa theory, e.g., in Emerton--Pollack--Weston \cite{EmPoWe2006} and Skinner--Urban \cite{SkUr2014} (and many others). The definition of ordinariness that we adopt in this paper is, however, different from (but closely related to) Hida's definition. Roughly speaking, Hida's definition of $p$-ordinariness only requires $a_p(f)$ to be a $p$-adic unit under a prechosen embedding $K_f\hookrightarrow\ov{\Q_p}$, while our definition requires $a_p(f)$ to be a $p$-adic unit under all embeddings $K_f\hookrightarrow\ov{\Q_p}$. In particular, the two definitions coincide when $K_f=\Q$. The reason for us to adopt a different definition is that a weight $2$ newform $f$ is ordinary at $p$ under our definition if and only if the abelian variety $A_f/\Q$ associated to $f$ via the Eichler--Shimura theory is ordinary at $p$ (see \Cref{f-ordinary-iff-Af-ordinary}).

Under this definition, we obtain the following result on ordinary primes of modular forms.

\begin{theorem}
\label{modular-form-theorem}
Let $f$ be a newform of weight $2$ and level $N$ and let $K_f$ denote the coefficient field of~$f$. Suppose that $[K_f:\Q]=q$ or $2q$ for some prime~$q$. Then, the set of ordinary primes of $f$ over~$\Q$ has positive density.

More generally, if there exists a rational prime $p$ that remains inert in $K_f$ or splits into two distinct primes in $K_f$ of the same inertia degree, then the set of ordinary primes of $f$ over $\Q$ has positive density.
\end{theorem}

As suggested in \Cref{result-comparison} (4), we will first prove the following two results that are formulated independently of abelian varieties or modular forms. Here \Cref{galois-rep-theorem} concerns compatible systems of Galois representations and \Cref{inert-in-K-theorem} is merely on sequences of algebraic integers indexed by primes (resembling traces of Frobenius). We will then apply these results to deduce the main theorems for abelian varieties and modular forms. Specifically, \Cref{galois-rep-theorem} deals with the ``$2q$'' case of \Cref{ablian-var-theorem-2} and \Cref{modular-form-theorem}, and the (P2) case of \Cref{ablian-var-theorem-3}; \Cref{inert-in-K-theorem} deals with the ``$q$'' case of \Cref{ablian-var-theorem-2}, \Cref{abelian-var-quadratic-theorem}, and \Cref{modular-form-theorem}, and the (P1) case of \Cref{ablian-var-theorem-3}.

\begin{theorem}
\label{galois-rep-theorem}
Let $F$ and $K$ be two number fields, let $S$ be a finite set of primes of $F$, and let $(a_v)_{v\in\Sigma_F\setminus S}$ be a fixed sequence in $\cl{O}_K$. Let $\epsilon$ be a finite order Hecke character of $F$ with image lying in $K^\times$. Let $T$ be an infinite set of rational primes which split completely in $K$. For each rational prime $\ell\in T$, let $\rho_\ell:\Gal(\ov{F}/F)\rightarrow\GL_2(K\otimes\Q_\ell)$ be a continuous Galois representation. Suppose that 
\begin{enumerate}[label=\rm{(\arabic*)}]
    \item (ramification) for every $\ell\in T$, each $\rho_\ell$ is unramified at every $v\in\Sigma_F$ with $v\notin S$ and $v\nmid\ell$;
    \item ($\ell$-independence) for every $\ell\in T$ and every $v\in\Sigma_F$ with $v\notin S$ and $v\nmid\ell$, the characteristic polynomial of $\rho_\ell(\Frob_v)\in \GL_2(K\otimes \Q_\ell)$ is
    \begin{align*}
        X^2-(a_v\otimes 1)X+(\epsilon(v)\N(v)\otimes1),
    \end{align*}
    which is independent of $\ell$, where $\N(v)=|\cl{O}_F/v|$; 
    \item (the Ramanujan conjecture/Hasse--Weil bound) the sequence $(a_v)_{v\in\Sigma_F\setminus S}$ satisfies that
    \begin{align*}
        |\iota(a_v)|_\C\leq 2\N(v)^{1/2}
    \end{align*}
    for all $v\in\Sigma_F$ with $v\notin S$ and all embeddings $\iota:K\hookrightarrow\C$;
    \item (surjectivity) there exists two finite extensions of fields $F'/F$ and $K/K'$ such that for every $\ell\in T$, we have, up to conjugation, that
    \begin{align*}
        \rho_\ell(\Gal(\ov{F}/F'))=\{M\in\GL_2(\cl{O}_{K'}\otimes\Z_\ell) : \det M\in (\Z\otimes\Z_\ell)^\times (\simeq   \Z_\ell^{\times})\}. 
    \end{align*}
\end{enumerate}
Suppose further that there exists a rational prime $p$ that lies below a degree $1$ prime of $F$ and splits into two distinct primes in $K$ of the same inertia degree. Then, the set
\begin{align*}
    S^{\mathrm{ord}}:=\{v\in\Sigma_F:v\notin S\text{ and }\lambda\nmid (a_v)\text{ for all $\lambda\in\Sigma_K$ with }\lambda|\N(v)\}
\end{align*}
has positive density.
\end{theorem}

\begin{theorem}
\label{inert-in-K-theorem}
Let $F$ and $K$ be two number fields and let $S$ be a finite set of primes of $F$. Let $(a_v)_{v\in\Sigma_F\setminus S}$ be a fixed sequence in $\cl{O}_K$ satisfying that 
\begin{enumerate}[label=\rm{(\arabic*)}]
    \item (the Ramanujan conjecture) $|\iota(a_v)|_\C\leq 2\N(v)^{1/2}$ for all $v\in\Sigma_F$ with $v\notin S$ and all embeddings $\iota:K\hookrightarrow\C$;
    \item (non-lacunary) the set $\{v\in\Sigma_F:v\notin S\text{ and }a_v=0\}$ of primes of $F$ has density $0$.
\end{enumerate}
Suppose further that there exists a rational prime $p$ that lies below a degree $1$ prime of $F$ and remains inert in $K$. Then, the set
\begin{align*}
    S^{\mathrm{ord}}:=\{v\in\Sigma_F:v\notin S\text{ and }\lambda\nmid (a_v)\text{ for all $\lambda\in\Sigma_K$ with }\lambda|\N(v)\}
\end{align*}
 has positive density.
\end{theorem}

\begin{remark}
\label{galois-rep-theorem-remark}
We make two remarks on \Cref{galois-rep-theorem} and \Cref{inert-in-K-theorem}.
\begin{enumerate}[label=(\arabic*)]
    \item The surjectivity condition (4) in \Cref{galois-rep-theorem} is intentionally relaxed to avoid the need for assuming that the endomorphisms of the abelian varieties are defined over the base field or that the modular forms have no inner twists. 
    \item By assuming a condition similar to that in \Cref{ablian-var-theorem-3} on the existence of some rational prime with specific joint splitting behaviors, we can also apply \Cref{galois-rep-theorem} and \Cref{inert-in-K-theorem} to obtain similar positive density results on the ordinary primes of certain Hilbert newforms. The necessary ingredients for the proof in this case are given by \cite{Ca1986} and \cite{Ta1989} on the existence of the associated Galois representations, \cite[Appendix~B]{Ne2012} on the surjectivity of the Galois images, and \cite{Bl2006} on the Ramanujan conjecture.
\end{enumerate}
\end{remark}

We now end this section by sketching the main steps of the proof of \Cref{galois-rep-theorem} and its application to the (P2) case of \Cref{ablian-var-theorem-3}. The general approach to proving such results can be traced back to the work of Serre, Katz, and Ogus, and has been refined by Pink, Noot, and others, as previously mentioned. Indeed, our proof is based on similar ideas.
\begin{enumerate}[leftmargin=1.5cm, label=Step \arabic*.]
    \item By analyzing the splitting behaviors of primes in (non-Galois) extensions of number fields, construct a positive density set $S_0$ of primes $v$ of $F$ that are of degree $1$ and lie above rational primes $p$ that split into two distinct primes in $K$ of the same inertia degree.
    \item Associate to each prime $v\in S_0$ the value $\N_{K/\Q}(a_v)^2\cdot \N(v)^{-g}$ and use the splitting behavior of $p$ and the Hasse--Weil bound of $a_v$ to show that for $v\notin S^{\ord}$, $\N_{K/\Q}(a_v)^2\cdot \N(v)^{-g}$  belongs to a finite set of integers (with size at most $2^{2g}+1$).
    \item Use the surjectivity assumption of the residual Galois representations and the Chebotarev density theorem to show that $\{v:\N_{K/\Q}(a_v)^2\cdot \N(v)^{-g}=c\}$ for a fixed integer $c$ has density $0$, thus concluding that the density of the primes $v\in S_0$ with $v\notin S^{\ord}$ is 0.
    \item Prove a surjectivity result on the Galois representations attached to the relevant abelian varieties in order to apply \Cref{galois-rep-theorem} to the (P2) case of \Cref{ablian-var-theorem-3}.
\end{enumerate}

\subsection*{Overview} In \Cref{sec:lem-NF}, we collect results concerning the splitting types of primes in number field extensions. In \Cref{section-galois-rep}, we review the essential background and results related to the Galois representations of abelian varieties and modular forms. Subsequently, in \Cref{sec:easier-thm} and \Cref{sec:harder-thm}, we prove \Cref{inert-in-K-theorem} and \Cref{galois-rep-theorem}, respectively. The proof of the former relies solely on results from \Cref{sec:lem-NF}, while the latter involves also a detailed analysis of the images of the Galois representations. Lastly, in \Cref{sec:application}, we present the proofs of \Cref{ablian-var-theorem-2},  \Cref{abelian-var-quadratic-theorem}, \Cref{ablian-var-theorem-3}, and \Cref{modular-form-theorem} as applications of \Cref{galois-rep-theorem} and \Cref{inert-in-K-theorem}. Finally, in the Appendix, we provide an independent proof of the density of ordinary primes (as defined in this paper) for CM modular forms of weight 2.

\subsection*{Notation}
Throughout this paper, $A$ usually denotes an abelian variety and $f$ usually denotes a modular form. 

In the case of abelian varieties, $g$ usually denotes the dimension of $A$, $F$ usually denotes the field of definition of $A$, and $K$ usually denotes the number field embedded into the $F$-endomorphism algebra of $A$. When enlarging the field of definition of $A$, we usually use $F'$ to denote the enlarged field of definition and $K'$ to denote the number field embedded into the center of the $F'$-endomorphism algebra of $A$ (and usually $K'\subseteq K$).

In the case of modular forms, $K_f$ usually denotes the coefficient field of $f$ and $g$ usually denotes the degree of $K_f$. We usually use the notations $F_f'/\Q$ and $K_f/K_f'$ for finite extensions of fields, which are needed when discussing the Galois images of the modular forms with inner twists. 

For any number field $E$ (e.g., $F,F',K,K'$), $\cl{O}_E$ denotes the ring of integers of $E$, $\Sigma_E$ denotes the set of nonzero prime ideals of $E$,  $\N_{E/\Q}(\cdot)$ denotes the absolute norm of an element of $E$, and $\N(\cdot)$ denotes the ideal norm.  Specifically for $F$ (resp.~$F'$) and $K$ or $K_f$ (resp.~$K'$), we use $v$ (resp.~$v'$) to denote a prime of $F$ (resp.~$F'$), use $p$ to denote the rational prime that lies below $v$ (resp.~$v'$), use $\lambda$ (resp.~$\lambda'$) to denote a prime of $K$ or $K_f$ (resp.~$K'$), and use $\ell$ to denote the rational prime that lies below~$\lambda$ (resp.~$\lambda'$).

\section{Lemmas on primes and number fields}\label{sec:lem-NF}

We first begin by proving some lemmas, mainly on the splitting behaviors of primes in (non-Galois) field extensions. 
In particular, \Cref{one-implies-positive-density} (resp.~\Cref{field-of-degree-q-or-2q}) is crucial to constructing a positive density set of primes that are ordinary in \Cref{galois-rep-theorem} and \Cref{inert-in-K-theorem} (resp.~\Cref{ablian-var-theorem-2}, \Cref{abelian-var-quadratic-theorem}, and \Cref{modular-form-theorem}).

\subsection{Correspondence between places and embeddings}

\begin{lemma}
\label{places-qlbarembeddings}
Let $K/\Q$ be a number field and let $\ell$ be a rational prime. For each embedding $\tau:K\hookrightarrow\ov{\Q_\ell}$, define a place $\lambda_\tau$ of $K$ by 
\begin{align*}
    |x|_{\lambda_\tau}:=|\tau(x)|_\ell
\end{align*}
for all $x\in K$, where $|\cdot|_\ell$ is the $\ell$-adic absolute value on $\ol{\Q_\ell}$ with $|\ell|_\ell=\ell^{-1}$. Then, the assignment $\tau\mapsto\lambda_\tau$ gives a one-to-one correspondence between embeddings $\tau:K\hookrightarrow\ov{\Q_\ell}$ up to 
composition  
with elements in $\Gal(\ov{\Q_\ell}/\Q_\ell)$ and places $\lambda|\ell$ of $K$.
\end{lemma}
\begin{proof}
See (8.1) Extension Theorem in \cite[Chapter~II]{Ne1999}.
\end{proof}

\begin{lemma}
\label{places-qbarembeddings}
Let $K/\Q$ be a number field and let $\ell$ be a rational prime that splits completely in $K$. Fix an embedding $\iota_\ell:\ov{\Q}\hookrightarrow\ov{\Q_\ell}$. For each embedding $\sigma:K\hookrightarrow\ov{\Q}$, define a place $\lambda_\sigma$ of $K$ by 
\begin{align*}
    |x|_{\lambda_\sigma} :=|\iota_\ell(\sigma(x))|_\ell 
\end{align*}
for all $x\in K$. Then, the assignment $\sigma\mapsto\lambda_\sigma$ defines a one-to-one correspondence between embeddings $\sigma:K\hookrightarrow\ov{\Q}$ and places $\lambda|\ell$ of $K$. 
\end{lemma}

\begin{proof}
By \Cref{places-qlbarembeddings}, it suffices to show that the assignment $\sigma\mapsto\iota_\ell\circ\sigma$ gives a one-to-one correspondence between embeddings $\sigma:K\hookrightarrow\ov{\Q}$ and embeddings $\tau:K\hookrightarrow\ov{\Q_\ell}$ up to
composition with elements in $\Gal(\ov{\Q_\ell}/\Q_\ell)$. Let $\sigma:K\hookrightarrow\ov{\Q}$ be an embedding and let $\tau=\iota_\ell\circ\sigma$. Since $\ell$ splits completely in $K$, the completion $\widehat{\tau(K)}$ of $\tau(K)$ in $\ov{\Q_\ell}$ is thus the same as $\Q_\ell$. 
In this way, $\tau$ factors as
\begin{align*}
    \tau:K\simarrow\tau(K)\subseteq\widehat{\tau(K)}=\Q_\ell\subseteq\ov{\Q_\ell},
\end{align*}
so $\alpha\circ\tau=\tau$ for any $\alpha\in\Gal(\ov{\Q_\ell}/\Q_\ell)$.

Now, suppose that $\sigma_1,\sigma_2:K\hookrightarrow\ov{\Q}$ satisfy that $\iota_\ell\circ\sigma_1=\alpha\circ\iota_\ell\circ\sigma_2$ for some $\alpha\in\Gal(\ov{\Q_\ell}/\Q_\ell)$. Then, $\iota_\ell\circ\sigma_1=\iota_\ell\circ\sigma_2$ by the previous discussion, and hence $\sigma_1=\sigma_2$. This implies that the assignment $\sigma\mapsto\iota_\ell\circ\sigma$ is injective, and hence it is bijective  by \Cref{places-qlbarembeddings} since $\ell$ splits completely in $K$.
\end{proof}

\subsection{Splitting of primes}
We will now review and prove some results on the splitting of primes in non-Galois extensions. The results in this section guarantee that it suffices to assume the existence of one rational prime satisfying the suitable splitting behaviors assumed in \Cref{galois-rep-theorem} and \Cref{inert-in-K-theorem}, respectively. 

Before proving the results, let us first define some necessary notations. For a set $X$, let $\mathrm{Sym}(X)$ denote the symmetric group on $X$. For any $n\in\Z^+$, let $S_n$ denote the symmetric group of degree $n$, i.e., $S_n:=\mathrm{Sym}(\{1,2,\ldots,n\})$. We will also adopt the usual cycle notation, i.e., $(a_1,a_2,\ldots,a_r)$, to represent elements in $S_n$.

\begin{lemma}
\label{embed-into-symmetric}
Let $K/E$ be a finite extension of number fields and let $\wt{K}/E$ be the Galois closure of $K/E$. Let $G=\Gal(\wt{K}/E)$, $H=\Gal(\wt{K}/K)$, and $\Sigma=\{\sigma:K\rightarrow \wt{K}\;\big|\;\sigma|_E=\id_E\}$. Let $G$ (and hence $H$) act on $\Sigma$ via 
composition. Write $n=[K:E]$. Then, 
\begin{enumerate}[label=\rm{(\arabic*)}]
    \item $|\Sigma|=n$ and $G$ acts transitively on $\Sigma$.
    \item The action of $G$ on $\Sigma$ induces an embedding $G\hookrightarrow\mathrm{Sym}(\Sigma)$. 
    \item Let $\sigma_n\in\Sigma$ be the unique element such that $\sigma_n=\id_K$. 
    Then, $H=\Stab_G(\{\sigma_n\})$ and the action of $H$ on $\Sigma\backslash\{\sigma_n\}$ induces an embedding $H\hookrightarrow\mathrm{Sym}(\Sigma\backslash\{\sigma_n\})$.
    \item $H\backslash G\simeq\Sigma$ as sets with $G$-action.
\end{enumerate}
\end{lemma}
\begin{remark}
If we identify $\Sigma$ with $\{1,\ldots,n\}$ such that $\sigma_n$ is identified with $n$, then it follows from the lemma that $G$ can be identified with a subgroup of $S_n$ 
and $H=\Stab_G(\{n\})$ can be identified with a subgroup of $S_n$ consisting of cycles not containing $n$.
\end{remark}
\begin{proof}
As $K/E$ is a finite extension of number fields, there exists $\alpha\in K$ such that $K=E(\alpha)$. Let $m_\alpha\in E[X]$ be the minimal polynomial of $\alpha$ over $E$. Then, $m_\alpha$ is of degree $n$ and has exactly $n$ distinct roots. There is then a one-to-one correspondence between elements $\sigma\in\Sigma$ and roots $\sigma(\alpha)$ of $m_\alpha$. Hence, $|\Sigma|=n$. In particular, the action of $G$ on $\Sigma$ corresponds to the action of $G$ on the set of roots of $m_\alpha$, which is transitive.

Now, we will show that $G$ acts faithfully on $\Sigma$ so that $G\hookrightarrow\mathrm{Sym}(\Sigma)$. Let $g_1,g_2\in G$ with $g_1\sigma=g_2\sigma$ for all $\sigma\in\Sigma$. Then, $g_1(\sigma(\alpha))=g_2(\sigma(\alpha))$ for all $\sigma\in\Sigma$. Hence, $g_1=g_2$ on all roots of~$m_\alpha$. Since $\wt{K}$ is the splitting field of $m_\alpha(X)$ over $E$, it follows that $g_1=g_2$.

For the third part, it suffices to show that $H=\Stab_G(\{\sigma_n\})$ so that it follows from the faithful action that $H\hookrightarrow\mathrm{Sym}(\Sigma\backslash\{\sigma_n\})$. Note that any element $\sigma\in\Sigma$ is uniquely determined by $\sigma(\alpha)$. Hence, $g\in\Stab_G(\{\sigma_n\})$ is equivalent to $g(\sigma_n(\alpha))=\sigma_n(\alpha)$, i.e., $g(\alpha)=\alpha$. As $K=E(\alpha)$, $g(\alpha)=\alpha$ is then equivalent to $g|_K=\id_K$, which is equivalent to $g\in H$. Hence, $H=\Stab_G(\{\sigma_n\})$.

The last part simply follows from the fact from group theory that if $G$ acts transitively on $\Sigma$, then $\Stab_G(\{\sigma\})\backslash G\simeq\Sigma$ as sets with $G$-action for all $\sigma\in\Sigma$.
\end{proof}

Let $K/E$ be an extension of number fields of degree $n$. Given a partition $\cl{P}=(a_1,\ldots,a_r)$ of~$n$, i.e., $n=\sum_{i=1}^ra_i$ with $a_i\in\Z^+$ for each $1\leq i\leq r$, we say that a prime $\fr{p}$ of $E$ is \emph{of splitting type $\cl{P}$ in $K$} if it is unramified in $K$ and $\fr{p}=\wp_1\cdots\wp_r$ in~$K$ where the inertia degree of each $\wp_i$ is $a_i$. On the other hand, given a partition $\cl{P}=(a_1,\ldots,a_r)$ of~$n$, we say that an element $g\in S_n$ is \emph{of cycle type $\cl{P}$} if it is a product of disjoint cycles of length $a_1,\ldots,a_r$.

Let us also recall two facts from the theory of field extensions. Let $K/E$ and $K^\prime/E$ be two finite extensions of number fields. Then, any prime of $E$ that is unramified in both $K$ and $K^\prime$ stays unramified in $KK^\prime$. In particular, any prime of $E$ that is unramified in $K$ stays unramified in the Galois closure. Also, if $K/E$ and $K^\prime/E$ are Galois extensions, then so is $KK^\prime/E$.

\begin{lemma}
\label{splitting-of-prime}
Let $L/K/E$ be a tower of finite extensions of number fields such that $L/E$ is a Galois extension. Let $G=\Gal(L/E)$ and $H=\Gal(L/K)$. Let $p$ be a prime of $E$ that is unramified in $L/E$ and let $\fr{p}$ be a prime of $L$ lying above $p$. Let $G_\fr{p}=\{\sigma\in G:\sigma(\fr{p})=\fr{p}\}$ and define an action of $G_\fr{p}$ on $H\backslash G$ by $H\sigma\cdot g=H\sigma g$. Then, 
\begin{enumerate}[label=\rm{(\arabic*)}]
    \item there is a one-to-one correspondence between orbits of $H\backslash G$ under the action of $G_\fr{p}$ and primes of $K$ that lie above $p$;
    \item the inertia degree of the prime of $K$ corresponding to an orbit of $H\backslash G$ is the size of the orbit.
\end{enumerate}
\end{lemma}
\begin{proof}
See \cite[Proposition~2.7, Chapter~III]{Ja1996}.
\end{proof}

\begin{lemma}
\label{splitting-of-prime-symmetric}
Let $K/E$ be an extension of number fields of degree $n$ and let $\wt{K}/E$ be the Galois closure of $K/E$. Identify $\Sigma=\{\sigma\colon K\rightarrow \wt{K}\;:\;\sigma|_E=\id_E\}$ with $\{1,\ldots,n\}$, $G=\Gal(\wt{K}/E)$ with a subgroup of $S_n$, and $H=\Gal(\wt{K}/K)$ with $\Stab_G(\{n\})$. Let $\cl{P}=(a_1,\ldots,a_r)$ be a partition of~$n$. Let $\fr{p}$ be a prime of $E$ that is unramified in $K$ and let $g\in G$ be a Frobenius lift at~$\fr{p}$. Then, $g$ is of cycle type $\cl{P}$ if and only if $\fr{p}$ is of splitting type $\cl{P}$ in $K$.

In particular, $G$ contains an element of cycle type $\cl{P}$ if and only if there exists a prime of $E$ that is of splitting type $\cl{P}$ in $K$.
\end{lemma}
\begin{proof}
As $g$ is a Frobenius lift at $\fr{p}$ and $\fr{p}$ is unramified in $K$ and hence in $\wt{K}$, we may identify the decomposition group $G_\fr{p}\subseteq G$ at $\fr{p}$ with~$\langle g\rangle$. Now, $H\backslash G\simeq\Sigma\simeq\{1,\ldots,n\}$ as sets with $G$-action by \Cref{embed-into-symmetric}. In order to obtain the splitting behavior of $\fr{p}$ in $K/E$, by \Cref{splitting-of-prime}, it suffices to look at the sizes of orbits of $\{1,\ldots,n\}$ under the action of $G_\fr{p}=\langle g\rangle$. It is now easy to see that the following are equivalent:
\begin{enumerate}[label=(\roman*)]
    \item $g$ is of cycle type $\cl{P}$;
    \item the sizes of orbits of $\{1,\ldots,n\}$ under the action of $G_\fr{p}=\langle g\rangle$ are exactly given by $\cl{P}$;
    \item $\fr{p}$ is of splitting type $\cl{P}$.
\end{enumerate}
\end{proof}

\begin{proposition}
\label{one-implies-positive-density}
Let $K_1/E,\ldots, K_m/E$ be finite extensions of number fields with degree $n_1,\ldots,n_m$ respectively. Let $\cl{P}_i$ be a partition of $n_i$. Then, the following are equivalent:
\begin{enumerate}[label=\rm{(\roman*)}]
    \item there exists a prime of $E$ that is of splitting type $\cl{P}_i$ in $K_i$ for all $i$;
    \item the set of primes of $E$ that are of splitting type $\cl{P}_i$ in $K_i$ for all $i$ has positive density.
\end{enumerate}
\end{proposition}
\begin{proof}
Clearly, (ii) implies (i), so suppose that (i) holds. Let $\wt{K_i}/E$ be the Galois closure of $K_i/E$ and let $L=\wt{K_1}\cdots\wt{K_m}$ be the composite field. Note that $L/E$ is Galois since each $\wt{K_i}/E$ is Galois. Let $G_i=\Gal(\wt{K_i}/E)\hookrightarrow S_{n_i}$ and let $G=\Gal(L/E)$. Let $\pi_i:G\twoheadrightarrow G_i$ denote the projection map. Note that
\begin{align*}
    \Sigma_i:=\{\sigma\colon K_i\rightarrow L\;:\;\sigma|_E=\id_E\}=\{\sigma\colon K_i\rightarrow \wt{K_i}\;:\;\sigma|_E=\id_E\},
\end{align*}
and the map $G\rightarrow S_{n_i}$ given by the action of $G^\prime_i$ on $\Sigma_i$ factors as $G\twoheadrightarrow G_i\hookrightarrow S_{n_i}$.

Let $\fr{p}$ be a prime of $E$ that is of splitting type $\cl{P}_i$ in $K_i$ for each $i$. Since $\fr{p}$ is unramified in $K_i$ for all $i$, $\fr{p}$ is also unramified in $L$. Let $G_\fr{p}=\langle g\rangle\subseteq G$ be the decomposition group at $\fr{p}$, where $g$ is a Frobenius lift at $\fr{p}$, and let $C_g$ be the conjugacy class of $g$ in $G$. By the Chebotarev density theorem, there exists a positive density set of primes of $E$ that are unramified in $L/E$ and whose Frobenius elements lie in~$C_g$. Let $\fr{q}$ be one of them and we may identify $G_\fr{q}$ with $\langle g\rangle$ and the Frobenius lift at $\fr{q}$ with $g$. By \Cref{splitting-of-prime-symmetric}, since $\fr{p}$ is of splitting type $\cl{P}_i$ in $K_i$, each $\pi_i(g)\in G_i$ is of cycle type $\cl{P}_i$. It then follows from \Cref{splitting-of-prime-symmetric} again that $\fr{q}$ is of splitting type $\cl{P}_i$ in $K_i$, whence (ii) follows.
\end{proof}

\subsection{Extensions of degree $q$ or $2q$}

We will now focus on the splitting behaviors of primes in field extensions of degree $q$ or $2q$ for some prime $q$, which is one of the ingredients for the proof of \Cref{ablian-var-theorem-2} and \Cref{modular-form-theorem}. We will first prove a lemma on transitive subgroups of $S_{2q}$ and then prove the important proposition on field extensions of degree $q$ or $2q$.

\begin{lemma}
\label{transitive-2q}
Let $q$ be a prime and $G\hookrightarrow S_{2q}$ be a transitive subgroup, i.e., $G$ acts transitively on $\{1,2,\ldots,2q\}$. Then, $G$ contains an element that is a product of two disjoint cycles of length $q$.
\end{lemma}
\begin{proof}
Let $G\cdot\{2q\}$ be the $G$-orbit of $2q$. Note that  $|G\cdot \{2q\}|=2q$. Since $G$ is transitive, $2q$ divides $|G|$ by the orbit-stabilizer theorem. Also, $|G|$ divides $|S_{2q}|=(2q)!$. Hence, $1\leq v_q(|G|)\leq 3$.

\textbf{Case 1:} $v_q(|G|)=3$. Then, $q=2$ and $G$ is a transitive subgroup of $S_4$ with $8$ dividing $|G|$. By the classification of subgroups of $S_4$, $G$ is either $S_4$ or conjugate to $\langle(1,2,3,4),(1,3)\rangle$ in $S_4$. In both cases, the element $(1,3)(2,4)$ (or its conjugate in $S_4$) lies in $G$, which satisfies the requirement.

\textbf{Case 2:} $v_q(|G|)=2$. By Sylow theorems, there exists a Sylow $q$-subgroup $I$ in $G$ of order~$q^2$. If $q=2$, then $I$ can only be conjugate to $\langle(1,2,3,4)\rangle$, $\langle(1,2),(3,4)\rangle$, or $\langle(1,2)(3,4), (1,3)(2,4)\rangle$ in $S_4$. In particular, $I$ always contains an element which is a product of two disjoint cycles of length~$q$. Now, suppose that $q\neq 2$. Note that $\langle(1, 2,\ldots, q),(q+1, \ldots, 2q)\rangle$ is a Sylow $q$-subgroup in $S_{2q}$ of order $q^2$, so $I$ is conjugate to $\langle(1, 2,\ldots, q),(q+1, \ldots, 2q)\rangle$ in $S_{2q}$ by Sylow theorems. Hence, $I$ is generated by two single cycles of length $q$ with disjoint elements. In particular, $I$ contains an element which is a product of two disjoint cycles of length $q$.

\textbf{Case 3:} $v_q(|G|)=1$. As $q$ is a prime, $G$ contains an element $\tau$ of order $q$.  Note that an element of order $q$ in $S_{2q}$ is either a single cycle of length $q$ or a product of two disjoint cycles of length~$q$. Suppose for the sake of contradiction that $\tau$ is a single cycle of length $q$. Then, there exists $a\in\{1,2,\cdots,2q\}$ such that $\tau(a)=a$. Since $G$ is transitive, there exists $\eta\in G$ such that $\eta(a)=2q$. Then, $(\eta\tau\eta^{-1})(2q)=2q$, so $\eta\tau\eta^{-1}\in \Stab_G(\{2q\})=:H$. By the orbit-stabilizer theorem, we have $|G\cdot\{2q\}|\cdot|H|=|G|$, so $q\nmid |H|$ as $v_q(|G|)=1$. This contradicts that $H$ contains the element $\eta\tau\eta^{-1}$ of order $q$.
\end{proof}

\begin{proposition}
\label{field-of-degree-q-or-2q}
Let $q$ be a prime and let $K/E$ be a finite extension of number fields. Then,
\begin{enumerate}[label=\rm{(\arabic*)}]
    \item if $[K:E]=q$, then there exists a prime of $E$ that is inert in $K$;
    \item if $[K:E]=2q$, then there exists a prime of $E$ that splits into two distinct primes in $K$ of the same inertia degree.
\end{enumerate}
\end{proposition}
\begin{proof}
Let $\wt{K}/E$ be the Galois closure of $K/E$. Write $n=[K:E]$, $G=\Gal(\wt{K}/E)$, and view $G\subseteq S_n$ by \Cref{embed-into-symmetric}. Suppose first that $[K:E]=q$. By \Cref{splitting-of-prime-symmetric}, it suffices to show that $G$ contains a cycle of length $q$. Note that $q$ divides $|G|$ since $G$ is a transitive subgroup of $S_q$. As $q$ is a prime, $G$ contains an element of order $q$, which is necessarily a cycle of length $q$ since $G\subseteq S_q$. Now, suppose that $[K:E]=2q$. By \Cref{splitting-of-prime-symmetric}, it suffices to show that $G$ contains an element that is a product of two disjoint cycles of length $q$, and indeed this simply follows from \Cref{transitive-2q}.
\end{proof}

\begin{corollary}
\label{quadratic-prime-splitting}
Let $F/\Q$ be a quadratic extension and $K/\Q$ be an extension of degree $q$ for some rational prime $q\geq 3$. Then, there exists a rational prime that splits completely in $F$ and remains inert in $K$.
\end{corollary}
\begin{proof}
Since $q$ is an odd prime, the composite field $FK$ is of degree $2q$. By \Cref{field-of-degree-q-or-2q}, there exists a  rational prime $p$ that splits into two distinct primes in $FK$ of the same inertia degree $q$. Then, $p$ splits into at most two primes in $K$. Suppose that $p=\fr{p}_1\fr{p}_2$ in $K$ of inertia degree $f_1$ and $f_2$, respectively. By the splitting behavior of $p$ in $FK$, both $\fr{p}_1$ and $\fr{p}_2$ remain inert in $FK$, so their inertia degrees as primes in $FK$ are $2f_1$ and $2f_2$, respectively. This contradicts that the inertia degree of the two primes in $FK$ that lie above $p$ is an odd prime $q$. Hence, $p$ must remain inert in~$K$, and as $p$ splits into two primes in $FK$, $p$ must split completely in $F$ since $F/\Q$ is quadratic. The result then follows. 
\end{proof}

\section{Results on Galois representations}
\label{section-galois-rep}

\subsection{Galois representations of abelian varieties}
\label{section-rep-abelian}

In this subsection, we will introduce the general theory of $\ell$-adic and $\lambda$-adic Galois representations of abelian varieties with extra endomorphisms. We will then specialize to the case of ``$\GL_2$-type'' abelian varieties, prove a result on the images of their Galois representations, and give a criterion, distinct from the definition, to identify their ordinary primes.

Let $A$ be an abelian variety of dimension $g$ and conductor $N_A$ defined over a number field $F$. Fix a rational prime~$\ell$. Let $T_\ell(A) := \varprojlim_n A[\ell^n]$ be the $\ell$-adic Tate module of $A$,
where $A[\ell^n]$ denotes the subgroup of $\ell^n$-torsion  points of $A(\overline{F})$. Let $V_\ell(A):=T_\ell(A)\otimes \Q_\ell$. Then, $V_\ell(A)$ is a $2g$-dimensional $\Q_\ell$-vector space, and the action of the absolute Galois group $\Gal(\ol{F}/F)$ on $V_\ell(A)$ induces a continuous Galois representation
\begin{align*}
    \rho_{A, \ell}: \Gal(\overline{F}/F) \to \Aut(V_\ell(A)) \simeq \GL_{2g}(\Q_\ell),
\end{align*}
which we call the $\ell$-adic Galois representation of $A$.\footnote{Due to the symplectic pairing on $V_\ell(A)$ coming from the polarization of $A$, the image of $\rho_{A,\ell}$ actually lies in $\GSp_{2g}(\Q_\ell)$.} This Galois representation satisfies that for all primes $v$ of $F$ with $v\nmid N_A\ell$, $\rho_{A,\ell}$ is unramified at $v$, and the characteristic polynomial
\begin{align}\label{eq:P_A}
    P_{A,v}(X):=\det(XI-\rho_{A, \ell}(\Frob_{v}))
\end{align}
is a degree $2g$ polynomial with integer coefficients, which is independent of the choice of $\ell$. By the Weil conjectures, for each embedding $\iota:K\rightarrow\C$ and each root $\alpha$ of $P_{A,v}(X)$, we have $|\iota(\alpha)|_\C=\N(v)^{1/2}$.

Now, we will discuss the case when $A$ has extra endomorphisms. The main reference for this part is \cite[II]{Ri1976}. Let $\End_F(A)$ denote the ring of $F$-endomorphisms of $A$ and let $\End_{\overline{F}}(A)$ denote the ring of geometric endomorphisms of $A$. Assume that $\End_F(A)\otimes\Q$ contains a number field~$K$. Then, it is known that $n=\frac{2g}{[K:\Q]}$ is necessarily an integer. In this case, we say that $A/F$ is \emph{of $\GL_n(K)$-type} or simply \emph{of $\GL_n$-type} without referring to the number field.

Suppose from now on that $A/F$ is of $\GL_n(K)$-type. Then, the extra endomorphisms induce more restrictions on the Galois image. Specifically, if we set $K_\ell:=K\otimes\Q_\ell$, then the Galois action on $V_\ell(A)$ is $K_\ell$-linear, and moreover $V_\ell(A)$ is a free $K_\ell$-module of rank $n$. Hence, the $\ell$-adic Galois representation can be interpreted as
\begin{align*}
     \rho_{A, \ell}: \Gal(\overline{F}/F) \to \Aut_{K_\ell}(V_\ell(A))\simeq \GL_n(K_\ell)=\GL_n(K\otimes\Q_\ell).\footnotemark
\end{align*}
\footnotetext{The shape of the Galois representation also explains the name ``$\GL_n$-type''.}
Since $K_\ell=\prod_{\lambda|\ell}K_\lambda$, for each prime $\lambda|\ell$ of $K$, we can put
\begin{align*}
    V_\lambda(A):=V_\ell(A)\otimes_{K_\ell}K_\lambda.
\end{align*}
Then, for each prime $\lambda$ of $K$, $V_\lambda(A)$ is an $n$-dimensional $K_\lambda$-vector space with the induced action of $\Gal(\overline{F}/F)$. Hence, each $V_\lambda(A)$ induces a continuous Galois representation
\begin{align*}
    \rho_{A,\lambda}:\Gal(\overline{F}/F) \to \Aut_{K_\lambda}(V_\lambda(A))\simeq \GL_n(K_\lambda),
\end{align*}
which we call the $\lambda$-adic Galois representation of $A$. Alternatively, $\rho_{A,\lambda}$ can be viewed as the composition of $\rho_{A,\ell}$ and the projection map $\GL_n(K\otimes\Q_\ell)\rightarrow\GL_n(K_\lambda)$. The Galois representations $\rho_{A,\ell}$ (under the new interpretation) and $\rho_{A,\lambda}$ (with $\lambda|\ell$) satisfy that for all primes $v$ of $F$ with $v\nmid N_A\ell$, $\rho_{A,\ell}$ and $\rho_{A,\lambda}$ are unramified at $v$, and the characteristic polynomials
\begin{align*}
    \det(XI-\rho_{A,\ell}(\Frob_v))\in(\cl{O}_K\otimes\Z)[X] \quad\text{ and }\quad\det(XI-\rho_{A,\lambda}(\Frob_v))\in\cl{O}_K[X]
\end{align*}
are of degree $n$ and independent of  $\ell$ and $\lambda$, where we used the interpretation $\rho_{A,\ell}: \Gal(\overline{F}/F)\to \GL_n(K\otimes \Q_\ell)$.
Moreover, it is known by \cite[Proposition 11.9, p. 586]{Sh1967} (see also \cite[Proposition 2.2, p. 322]{Chi1992}) that
\begin{align}
\label{char-poly-eqn}
    P_{A,v}(X)=\prod_{\sigma:K\rightarrow\C}\sigma\big(\det(XI-\rho_{A,\lambda}(\Frob_v))\big).
\end{align}

Let us now specialize to the case when $n=2$, i.e., when $A/F$ is of $\GL_2(K)$-type. For a prime $v$ of $F$ with $v\nmid N_A$ (and $v\nmid\N(\lambda)$), we write
\[
a_v(A)=\tr \rho_{{A,\lambda}}(\Frob_v) \quad\text{ and }\quad d_v(A)=\det \rho_{{A,\lambda}}(\Frob_v)
\]
so that
\begin{align*}
    \det(XI-\rho_{A,\lambda}(\Frob_v))=X^2-a_v(A)X+d_v(A)\in\cl{O}_K[X].
\end{align*}
To summarize, we have a family of Galois representations $\{\rho_{A,\ell}:\Gal(\ol{F}/F)\rightarrow\GL_2(K\otimes\Q_\ell)\}_{\ell\in\Sigma_\Q}$ such that for all rational primes $\ell$ and all primes $v$ of $F$ with $v\nmid N_A\ell$, $\rho_{A,\ell}$ is unramified at $v$, the characteristic polynomial of the Frobenius satisfies that 
\begin{align*}
    \det(XI-\rho_{A,\ell}(\Frob_v))=X^2-(a_v(A)\otimes 1)X+(d_v(A)\otimes 1)\in(\cl{O}_K\otimes\Z)[X]
\end{align*}
is of degree $2$ and independent of $\ell$, and for all embeddings $\iota:K\rightarrow\C$,
\begin{align*}
    |\iota(a_v(A))|_\C\leq 2\sqrt{\N(v)}.
\end{align*}
The last pieces of information about the Galois representations of $A$ needed for the proof of \Cref{ablian-var-theorem-3} are the determinant of $\rho_{A, \lambda}$ and image of $\rho_{A, \ell}$, which we present now.

The analysis of the shapes of the determinant and the Galois image depends heavily on \cite[Proposition~6.10]{FiFlGu2024} and \cite[Theorem~6.16]{BaGaKr2006}, and we will recall the former one here as we will need it in several places. We will also recall several definitions from \cite{FiFlGu2024} and \cite{BaGaKr2006} to make this part self-contained. Indeed, similar result with extra restrictions on the endomorphism algebra of $A$ was already established by Ribet \cite[Theorem 5.5.2, p. 801]{Ri1976} and reformulated by Lombardo \cite[Theorem 1.4 and Remark 1.6]{Lo2016}.

\begin{definition}[\cite{FiFlGu2024}]
We say that an abelian variety $A/F$ is \emph{of $\GL_n$-type} if there exists a number field $K\hookrightarrow\End_F(A)\otimes\Q$ with $n=\frac{2\dim A}{[K:Q]}$. Moreover, we say that $A/F$ is \emph{genuinely of $\GL_n$-type} if it is simple, of $\GL_n$-type, and its base change to $\ol{F}$ has no simple subvariety of $\GL_m$-type with $m<n$.
\end{definition}

\begin{definition}[\cite{FiFlGu2024}]
We say that a simple abelian variety $A/F$ is \emph{of the first kind} if $A$ is of type I, II, or III in Albert's classification, \footnote{For a detailed summary of the Albert's classification, see \cite[pp. 201--202]{Mu1970}.} or equivalently, the center of $\End_F(A)\otimes\Q$ is a totally real field. For an abelian variety $A/F$ that is genuinely of $\GL_n$-type, its base change to $\ol{F}$ is necessarily isogenous to $B^r$ for some simple abelian variety $B/\ol{F}$, and we say that $A/F$ is \emph{geometrically of the first kind} if $B$ is of the first kind.
\end{definition}

\begin{definition}[\cite{BaGaKr2006}]
 We say that an abelian variety  $A/F$ of dimension $g$ is {\emph{of class $\mathcal{A}$}} if  
 \begin{enumerate}
     \item $A/F$ is a simple, principally polarized abelian variety.
     \item $A/\ol{F}$ is of type I or II in Albert's classification and $D:=\End_{\ol{F}}(A)\otimes \Q=\End_{F}(A)\otimes \Q$.
     \item For every rational prime $\ell$, the  Zariski closure of the $\rho_{A, \ell}(\Gal(\ol{F}/F))$ in $\GL_{2g}/\Q_\ell$ is a connected algebraic group.
     \item $g=ehd$ where $h$ is an odd integer, $e$ is the degree of the  center $E$ of $D$, and $d^2=[D:E]$.
 \end{enumerate}
\end{definition}

\begin{proposition}[{\cite[Proposition~6.10]{FiFlGu2024}}]
\label{lambda-adic-rep-determinant}
Let $A/F$ be an abelian variety that is genuinely of $\GL_2$-type and geometrically of the first kind. Then, there exists a finite order character $\epsilon:\Gal(\ol{F}/F)\rightarrow K^\times$ such that for all primes $\lambda$ of $K$,
\begin{align*}
    \det\rho_{A,\lambda}=\epsilon\chi_\ell,
\end{align*}
where $\ell$ is the rational prime lying below $v$ and $\chi_\ell:\Gal(\ol{F}/F)\rightarrow K_\lambda^\times$ denotes the $\ell$-adic cyclotomic character. Moreover, if the center of $\End_F(A)\otimes\Q$ is a totally real field, then $\det\rho_{A,\lambda}=\chi_\ell$.
\end{proposition}

With all the tools in hand, we are now ready to prove the results on the determinant and the Galois image.

\begin{lemma}
\label{degree-2-char-poly}
Let $A/F$ be an abelian variety of dimension $g$ and conductor $N_A$ over a number field~$F$. Assume that $A$ is absolutely simple, $\End_F(A)\otimes\Q$ contains a number field $K$ of degree~$g$, and $A/\ol{F}$ is of the first kind. Then, there exists a finite order character $\epsilon:\Gal(\ol{F}/F)\rightarrow K^\times$ such that for all rational primes $\ell$, all primes $\lambda$ of $K$ with $\lambda|\ell$, and all primes $v$ of $F$ with $v\nmid N_A\ell$, 
\begin{align*}
    \det(XI-\rho_{A,\lambda}(\Frob_v))&=X^2-a_v(A)X+\epsilon(v)\N(v)\in\cl{O}_K[X] \\
    \det(XI-\rho_{A,\ell}(\Frob_v))&=X^2-(a_v(A)\otimes1)X+(\epsilon(v)\N(v)\otimes1)\in(\cl{O}_K\otimes\Z)[X],
\end{align*}
where $\rho_{A,\ell}$ is interpreted as $\Gal(\ol{F}/F)\rightarrow\GL_2(K\otimes\Q_\ell)$.
\end{lemma}
\begin{proof}
It is easy to check that the assumptions on $A$ guarantee that $A$ is genuinely of $\GL_2$-type and geometrically of the first kind. The result then simply follows from \Cref{lambda-adic-rep-determinant} (and our previous discussion).
\end{proof}

\begin{theorem}
\label{Galois-images}
Let $A/F$ be an abelian variety of dimension $g$ over a number field $F$. Assume that $A$ is absolutely simple, $\End_{\ol{F}}(A)\otimes\Q$ contains a number field $K$ of degree $g$,\footnote{Note that this is weaker than the assumption $K\hookrightarrow\End_F(A)\otimes\Q$ in \Cref{ablian-var-theorem-3}.} and the center of $\End_{\ol{F}}(A)\otimes\Q$ is a totally real field $K'$. Then, there exists a finite extension $F'/F$ such that for all sufficiently large prime $\ell$, we have, up to conjugation, that
\[
\rho_{A, \ell}(\Gal(\ol{F'}/F'))=\{M\in\GL_2(\cl{O}_{K'}\otimes\Z_\ell) : \det M\in  \Z_\ell^{\times}\}.
\]
\end{theorem}
\begin{proof}
 
If we replace $A$ by any abelian variety in its isogeny class over some finite extension of $F$, then the assumptions are still satisfied. Also, for all sufficiently large prime $\ell$, the image of the $\ell$-adic Galois representation remains unchanged over the field where the isogeny is defined. Since $A$ is isogenous over a finite extension of $F$ to a principally polarized abelian variety, we may assume that $A$ is principally polarized over $F$. 

Now, let $F'/F$ be a finite extension of number fields such that \footnote{The existence of this finite extension can be deduced from \cite[Theorem~4.1]{Si1992} and \cite[Lemma~2.7~and~Remark~3.1]{SiZa1995} (see also \cite{Se2013}).}
\begin{enumerate}[label=(\arabic*)]
    \item $D:=\End_{\ol{F}}(A)\otimes \Q=\End_{F'}(A)\otimes \Q$;
    \item for every rational prime $\ell$, the Zariski closure of $\rho_{A, \ell}(\Gal(\ol{F}/F'))$ in the algebraic group $\GL_{2g}/\Q_\ell$ is a connected algebraic group.
\end{enumerate}
By the assumptions and (2), $A/F'$ is simple and of $\GL_2$-type, and by \cite[Proposition~1.7]{Wu2020}, $A/F'$ is not of type III. As the center $K'$ of $D$ is assumed to be totally real, by \cite[Proposition~1.7]{Wu2020}, we have either $D=K'(=K)$ or $D$ is a division quaternion algebra over $K'$ with $[K':\Q]=g/2$. Putting all things together, it is then easy to check that $A/F'$ is of class $\cl{A}$, so by \cite[Theorem 6.16 (6.18)]{BaGaKr2006}, we have $\SL_2(\mathcal{O}_{K'}\otimes \Z_\ell) \subseteq \rho_{A, \ell}(\Gal(\ol{F'}/F'))$ for sufficiently large prime $\ell$. On the other hand, by the assumptions, $A/F'$ is (geometrically) of the first kind. As the center $K'$ of $D$ is totally real, by \Cref{lambda-adic-rep-determinant} (and gathering all primes $\lambda'$ of $K'$ with $\lambda'|\ell$), we have $\det\rho_{A,\ell}|_{\Gal(\ol{F}/F')}=\chi_\ell$ for all rational primes $\ell$.

Finally, for sufficiently large prime $\ell$, since 
\begin{equation}\label{eeq:image-inclusion}
\rho_{A, \ell}(\Gal(\ol{F'}/F'))\subseteq \{M\in\GL_2(\cl{O}_{K'}\otimes\Z_\ell) : \det M\in  \Z_\ell^{\times}\}
\end{equation}
and the exact sequence 
\[
0\to \SL_2(\mathcal{O}_{K'}\otimes \Z_\ell) \to  \{M\in\GL_2(\cl{O}_{K'}\otimes\Z_\ell) : \det M\in  \Z_\ell^{\times}\}\xrightarrow{ \det} \Z_\ell^{\times}\to 0
\]
also holds for $\rho_{A, \ell}(\Gal(\ol{F'}/F'))$, the inclusion (\ref{eeq:image-inclusion}) is actually an equality by the five lemma (see also \cite[Corollary~2.2]{Ri1975}). 
\end{proof}

To end this subsection, we will give a different criterion of ordinary primes of $A$ using the Frobenius trace $a_{v}(A)$ of the $\lambda$-adic representations. Here, we take one of the equivalent definitions of ordinary primes from \cite{Fi2024}.
\begin{definition}
     Let $A$ be an abelian variety of dimension $g$ and conductor $N_A$ defined over a number field $F$. We say that a prime $v$ of $F$ is an \emph{ordinary prime} of $A/F$ if $v\nmid N_A$ and $p\nmid a_{v,g}(A)$, where $p$ is the rational prime lying below $v$ and $a_{v,g}(A)$ is the coefficient of the term $X^g$ of $P_{A,v}(X)$ defined by \Cref{eq:P_A}.
\end{definition}

\begin{proposition}
\label{ordinary-lambda-trace}
Let $A/F$ be an abelian variety of dimension $g$ and conductor $N_A$ over a number field~$F$. Assume that $A$ is absolutely simple, $\End_F(A)\otimes\Q$ contains a number field $K$ of degree $g$, and $A/\ol{F}$ is of the first kind. Then, a prime $v\nmid N_A$ of $F$ is an ordinary prime of $A$ if and only if $\lambda \nmid (a_{v}(A))$ for all primes $\lambda|\N(v)$ of $K$.
\end{proposition}
\begin{proof}
It is an easy fact from algebraic number theory that for a rational prime $p$, $p\nmid\N_{K/\Q}(a_v(A))$ if and only if $\lambda\nmid (a_v(A))$ for all primes $\lambda|p$ of $K$. We will thus prove that $v$ is an ordinary prime of $A$ if and only if $p\nmid\N_{K/\Q}(a_v(A))$ for the rational prime $p$ lying below $v$.

By \Cref{char-poly-eqn},
\begin{align*}
     P_{A,v}(X)=\prod_{\sigma:K\rightarrow\C}(X^2-\sigma(a_v(A))X+\sigma(\epsilon(v))\N(v)).
\end{align*}
Since $\#\{\sigma:K\rightarrow\C\}=g$, it is easy to see that the coefficient $a_{v,g}(A)$ of the term $X^g$ is of the form
\begin{align*}
    a_{v,g}(A)=\prod_{\sigma:K\rightarrow\C}\sigma(a_v(A))+\N(v)\alpha=\N_{K/\Q}(a_v(A))+\N(v)\alpha
\end{align*}
for some algebraic integer $\alpha$. Note that both $a_{v,g}(A)$ and $\N_{K/\Q}(a_v(A))$ are integers, so $\alpha\in\Q$, and hence $\alpha\in\Z$. In particular, $p\nmid a_{v,g}(A)$ if and only if $p\nmid\N_{K/\Q}(a_v(A))$. The result then follows.
\end{proof}

\subsection{Galois representations of modular forms}
\label{galois-rep-of-modular-form-section}

In this subsection, we will introduce the $\ell$-adic Galois representations of modular forms and recall the properties of the images of those Galois representations. We will also briefly review the Eichler--Shimura theory and relate the $p$-ordinariness of weight $2$ modular forms to the $p$-ordinariness of the associated abelian varieties.

Let $f\in S_k(\Gamma_0(N), \epsilon)$ be a newform of weight $k$,  level $N$, and nebentypus $\epsilon$, i.e., $f$ is a normalized Hecke eigenform which is new in level $N$ and $\epsilon:\Z\rightarrow\C$ is any Dirichlet character of conductor dividing $N$.  Let $a_n(f)$ be the $n$-th Fourier coefficient of $f$ and let $K_f=\Q(\{a_n(f):(n,N)=1\})$ be the coefficient field of $f$. It is known that each $a_n(f)$ is an algebraic integer, $K_f$ is a number field, and $\cl{O}_{K_f}$ contains the image of $\epsilon$. Let $\ell$ be any rational prime. By the works of Eichler--Shimura for $k=2$, Deligne for $k\geq 2$ \cite{De1973}, and Deligne--Serre for $k=1$ \cite{DeSe1974}, there exists a continuous Galois representation
\begin{align*}
    \rho_{f, \ell}: \Gal(\ol{\Q}/\Q)\rightarrow\GL_2(K_f\otimes\Q_\ell)
\end{align*}
such that for every rational prime $p\nmid\ell N$, $\rho_{f,\ell}$ is unramified at $p$ and
\begin{align*}
    \det(XI-\rho_{f,\ell}(\Frob_p))=X^2-(a_p(f)\otimes 1)X+(\epsilon(p)p^{k-1}\otimes 1)\in (\cl{O}_{K_f}\otimes\Z)[X].
\end{align*}
By the Ramanujan conjecture, it is known that $|\iota(a_p(f))|_\C\leq 2p^{\frac{k-1}{2}}$ for all embeddings $\iota: K\hookrightarrow\C$.

As in the case of abelian varieties, the image  of $\rho_{f, \ell}$ is impacted by the ``inner twist" of $f$, which we introduce now. The main references for this part are \cite{Ri1980} and \cite{Ri1985}. The modular form $f$ is said to have \emph{complex multiplication (CM)} or called a \emph{CM modular form} if there exists a nontrivial Dirichlet character $\chi$ such that $a_p(f)=\chi(p)a_p(f)$ for all but finitely many primes $p$. More generally, the modular form $f$ is said to have \emph{inner twist} if there exists a nontrivial Dirichlet character $\chi$ and an embedding $\sigma:K_f\rightarrow\C$ such that
\begin{align*}
    \sigma(a_p(f))=\chi(p)a_p(f)
\end{align*}
for all but finitely many primes $p$. We also say that $f$ is a \emph{non-CM modular form} if $f$ has no complex multiplication.

Suppose that $f$ is a non-CM modular form. Then, for each embedding $\sigma:K_f\rightarrow\C$, there exists at most one Dirichlet character $\chi=\chi_\sigma$ such that $f$ has inner twist by $(\sigma,\chi)$.
All such embeddings $\sigma$ form a group $\Gamma_f$ and we set
\[
H_f: = \bigcap_{\sigma\in \Gamma_f}\Ker\chi_\sigma, \quad F_f':=\ol{\Q}^{H_f}, \quad K_f': = K_f^{\Gamma_f}.
\]
Consequently, we have the following big image  theorem for $\rho_{f, \ell}$, where $f$ possibly has inner twist. 
\begin{theorem}[{\cite[Theorem~3.1]{Ri1985}}]
We keep the notation above. Let $k\geq 2$ and $f\in S_k(\Gamma_0(N), \epsilon)$ be a non-CM newform. Then for all sufficiently large prime $\ell$, we have, up to conjugation, that
\[
\rho_{f, \ell}(\Gal(\ol{\Q}/F_f'))= \big\{M\in \GL_2(\mathcal{O}_{K_f'}\otimes \Z_\ell): \det M \in (\Z_\ell^\times)^{k-1}\big\}.
\]
\end{theorem}

We will now relate the ordinary primes of modular forms (of weight $2$) to the ones of abelian varieties. Let us first recall our definition of ordinary primes for modular forms.

\begin{definition}
Let $f$ be a newform of weight $k\geq 2$ and level $N$, let $a_n(f)$ be the $n$-th Fourier coefficient of $f$, and let $K_f$ be the coefficient field of $f$. For a rational prime $p\nmid N$ and a prime $\lambda|p$ of $K_f$, $f$ is called \emph{$\lambda$-ordinary} if $\lambda\nmid(a_p(f))$, and $f$ is called \emph{$p$-ordinary} if $f$ is $\lambda$-ordinary for all $\lambda|p$, or equivalently, $\lambda\nmid(a_p(f))$ for all primes $\lambda|p$ of $K_f$.
\end{definition}

\begin{remark}
It is an easy fact from algebraic number theory that for any number field $K$, $\alpha\in\cl{O}_K$, and rational prime $p$, $p|\N_{K/\Q}(\alpha)$ if and only if $\lambda|(\alpha)$ for some prime $\lambda|p$ of $K$. From this, one can give another equivalent definition that for a rational prime $p\nmid N$, $f$ is called \emph{$p$-ordinary} if $p\nmid\N_{K_f/\Q}(a_p(f))$.
\end{remark}

Let us now specialize to the case when $k=2$, i.e., $f\in S_2(\Gamma_0(N),\epsilon)$ is a newform of weight $2$. By the Eichler--Shimura theory, one can associate $f$ with an abelian variety $A_f/\Q$ \footnote{This abelian variety $A_f/\Q$ is constructed by taking a quotient of the Jacobian of the modular curve $X_1(N)$, and we refer to \cite[Section~6.6]{DiSh2005} and \cite{Ro1997} for a detailed review of the Eichler--Shimura theory.} of dimension $g=[K_f:\Q]$ such that $K_f=\End_{\Q}(A_f)\otimes\Q$ and that for all rational primes $p\nmid N$, $A_f$ has good reduction at $p$ and
\begin{align*}
    P_{A_f,p}(X)=\prod_{\sigma:K_f\rightarrow\C}\big(X^2-\sigma(a_p(f))X+\sigma(\epsilon(p))p\big),
\end{align*}
where $P_{A_f,p}(X)$ is the polynomial defined in \Cref{eq:P_A}.

\begin{proposition}
\label{f-ordinary-iff-Af-ordinary}
Let $f\in S_2(\Gamma_0(N),\epsilon)$ be a newform and let $A_f/\Q$ be the associated abelian variety. Then, for any rational prime $p\nmid N$, $f$ is ordinary at $p$ if and only if $A_f$ is ordinary at $p$.
\end{proposition}
\begin{proof}
Note that $f$ is ordinary at $p$ if and only if $p\nmid \N_{K_f/\Q}(a_p(f))$. We will prove that the latter is equivalent to $A_f$ being ordinary at $p$. For all primes $p\nmid N$, 
\begin{align*}
    P_{A_f,p}(X)=\prod_{\sigma:K_f\rightarrow\C}\big(X^2-\sigma(a_p(f))X+\sigma(\epsilon(p))p\big).
\end{align*}
Since $\#\{\sigma:K_f\rightarrow\C\}=[K_f:\Q]=:g$, it is easy to see that the coefficient $a_{p,g}(A_f)$ of the term $X^g$ is of the form
\begin{align*}
    a_{p,g}(A_f)=\prod_{\sigma:K_f\rightarrow\C}\sigma(a_p(f))+p\alpha=\N_{K_f/\Q}(a_p(f))+p\alpha
\end{align*}
for some algebraic integer $\alpha$. Note that both $a_{p,g}(A_f)$ and $\N_{K_f/\Q}(a_p(f))$ are integers, so $\alpha\in\Q$, and hence $\alpha\in\Z$. In particular, $p\nmid a_{p,g}(A_f)$ if and only if $p\nmid\N_{K_f/\Q}(a_p(f))$. The result then follows.
\end{proof}

\section{Proof of \Cref{inert-in-K-theorem}}\label{sec:easier-thm}

We will first give a proof of \Cref{inert-in-K-theorem} since its proof is much easier. We will adopt all the notations in \Cref{inert-in-K-theorem} unless stated otherwise. Recall that in this case, we have a rational prime~$p$ that lies below a degree $1$ prime of $F$ and remains inert in $K$.

\begin{lemma}
\label{p1-lemma}
Let $F$ and $K$ be two number fields and let $S$ be a finite set of primes of $F$. Let $(a_v)_{v\in\Sigma_F\setminus S}$ be a fixed sequence in $\cl{O}_K$ satisfying that
\begin{align*}
    |\iota(a_v)|_\C\leq 2\N(v)^{1/2}
\end{align*}
for all $v\in\Sigma_F$ with $v\notin S$ and all embeddings $\iota:K\rightarrow\C$. Let $v\notin S$ be a degree $1$ prime of $F$ and let $p$ be the rational prime lying below $v$. Suppose that $p\geq 5$ and $p$ is inert in $K$. Then, $a_v=0$ or $v\in S^{\ord}$.
\end{lemma}
\begin{proof}
Suppose that $v\notin S^{\ord}$. Then, there exists a prime $\lambda$ of $K$ with $\lambda|\N(v)$ such that $\lambda|(a_v)$, so $\N(\lambda)|\N((a_v))$ in $\Z$.
By the assumption on $(a_v)$,
\begin{align*}
    \N((a_v))=\N_{K/\Q}(a_v)=\bigg|\prod_{\sigma:K\hookrightarrow\ov{\Q}}\sigma(a_v)\bigg|=\prod_{\iota:K\rightarrow\C}|\iota(a_v)|_\C\leq 2^g\N(v)^{g/2},
\end{align*}
where we write $[K:\Q]=g$. By the assumptions on $v$ and $p$, we have $\N(v)=p$ and $\N(\lambda)=p^g$. Hence, $p^g|\N((a_v))$ with $\N((a_v))\leq 2^gp^{g/2}$. As $p\geq 5$, this implies that $\N((a_v))=0$ so $a_v=0$.
\end{proof}

We continue with the proof of \Cref{inert-in-K-theorem}. Let $T_1$ be the set of rational primes $p$ such that 
\begin{enumerate}[label=\rm{(\arabic*)}]
    \item $p$ lies below a degree $1$ prime of $F$;
    \item $p$ is inert in $K$;
    \item $p\geq 5$.
    \item $p$ does not lie below any prime in $S$. 
\end{enumerate}
By the assumption in \Cref{inert-in-K-theorem}, there exists a rational prime $p$ satisfying both (1) and (2), and by \Cref{one-implies-positive-density}, there exists positive density rational primes having the same splitting type as $p$ in both $F$ and $K$. In particular, these primes all satisfy both (1) and (2), implying that $T_1$ has positive density. 
Now, consider the set of primes of $F$ that lie above the primes in $T_1$, i.e., consider
\begin{align*}
    S_1:=\{v\in\Sigma_F: \N(v)=p\text{ for some }p\in T_1\}.
\end{align*}
By \Cref{p1-lemma}, if $v\in S_1$, then $a_v=0$ or $v\in S^{\ord}$. Since $\{v\in\Sigma_F:v\notin S\text{ and }a_v=0\}$ is assumed to have density $0$, $S_1$ thus has the same density as $S_1\cap\{v\in\Sigma_F:v\notin S\text{ and }a_v\neq 0\}\subseteq S^{\ord}$, so it suffices to prove that $S_1$ has positive density.

Note that each prime $p\in T_1$ gives at least one prime $v\in S_1$, so for $X>0$
\begin{align*}
    \#\{v\in S_1:\N(v)\leq X\}\geq\#\{p\in T_1:p\leq X\}.
\end{align*}
On the other hand, by Landau's prime ideal theorem (and the usual prime number theorem),
\begin{align*}
    \#\{v\in\Sigma_F:\N(v)\leq X\}\sim\frac{X}{\log X}\sim\#\{p:p\leq X\}.
\end{align*}
Hence,
\begin{align*}
    \mathrm{Density}(S_1)=\lim_{X\rightarrow\infty}\frac{\#\{v\in S_1:\N(v)\leq X\}}{\#\{v\in\Sigma_F:\N(v)\leq X\}}&\geq\lim_{X\rightarrow\infty}\frac{\#\{p\in T_1:p\leq X\}}{\#\{v\in\Sigma_F:\N(v)\leq X\}}\\
    &=\lim_{X\rightarrow\infty}\frac{\#\{p\in T_1:p\leq X\}}{\#\{p:p\leq X\}} \\
    &=\mathrm{Density}(T_1)>0.
\end{align*}
\Cref{inert-in-K-theorem} then follows.

\section{Proof of \Cref{galois-rep-theorem}}\label{sec:harder-thm}

We will now give a proof of \Cref{galois-rep-theorem}. We will adopt all the notations in \Cref{galois-rep-theorem} unless stated otherwise. Recall that in this case, we have a rational prime $p$ that lies below a degree~$1$ prime of $F$ and splits into two distinct primes in $K$ of the same inertia degree.

\subsection{A proof assuming \Cref{density-of-S(c)}}
\label{incomplete-proof-of-galois-rep-theorem}
The proof of \Cref{galois-rep-theorem} follows a similar approach to that of   \Cref{inert-in-K-theorem}. We first prove a lemma showing that under the assumption on the rational prime $p$,  almost every prime $v$ of $F$ above $p$ is either an ordinary prime or belongs to a special set. 
In the case of \Cref{inert-in-K-theorem},
the special set is $\{v\in\Sigma_F:v\notin S\text{ and }a_v = 0\}$, which has density 0 by assumptions. For \Cref{galois-rep-theorem}, the special set is $S(c)$ (see \Cref{eq:S(c)} for its precise definition) or, more accurately, a finite union of such sets. 
It is much trickier to show that the density of $S(c)$ is $0$. Therefore,  we  first proceed with  the proof assuming that $S(c)$ has density $0$, and then, in \Cref{density-of-S(c)-section}, we discuss its density properly. 

\begin{lemma}
\label{p2-lemma}
Let $F$ and $K$ be two number fields and let $S$ be a finite set of primes of $F$. Let $(a_v)_{v\in\Sigma_F\setminus S}$ be a fixed sequence in $\cl{O}_K$ satisfying that
\begin{align*}
    |\iota(a_v)|_\C\leq 2\N(v)^{1/2}
\end{align*}
for all $v\in\Sigma_F$ with $v\notin S$ and all embeddings $\iota:K\rightarrow\C$. Let $v\in\Sigma_F$ with $v\notin S$ be a degree $1$ prime of $F$ and let $p$ be the rational prime lying below $v$. Suppose that $p$ splits into two distinct primes in~$K$ of the same inertia degree. Then, $v\in S^{\ord}$ or $\N_{K/\Q}(a_v)^2=c\N(v)^g$ for some $c\in\Z$ with $0\leq c\leq 2^{2g}$.
\end{lemma}
\begin{proof}
Suppose that $v\notin S^{\ord}$. Then, there exists a prime $\lambda$ of $K$ with $\lambda|\N(v)$ such that $\lambda|(a_v)$, so $\N(\lambda)|\N((a_v))$ in $\Z$. By the assumption on $a_v$, 
\begin{align*}
    \N((a_v))=\bigg|\prod_{\sigma:K\hookrightarrow\ov{\Q}}\sigma(a_v)\bigg|=\prod_{\iota:K\rightarrow\C}|\iota(a_v)|_\C\leq 2^g\N(v)^{g/2},
\end{align*}
where we write $[K:\Q]=g$. By the assumptions on $v$ and $p$, we have $\N(v)=p$ and $\N(\lambda)=p^{g/2}$ ($g$ is necessarily even in this case). Hence, $p^{g/2}|\N((a_v))$ with $\N((a_v))\leq 2^gp^{g/2}$, implying that $\N_{K/\Q}(a_v)^2=c\N(v)^g$ for some $c\in\Z$ with $0\leq c\leq 2^{2g}$.
\end{proof}

Now, we continue with the proof. For $c\in\Z$, let
\begin{align}\label{eq:S(c)}
    S(c):=\{v\in\Sigma_F:v\notin S\text{ and }\N_{K/\Q}(a_v)^2=c\N(v)^g\}.
\end{align}
Suppose for now that $S(c)$ has density $0$, which will be proved in \Cref{density-of-S(c)}. We will show under this assumption that \Cref{galois-rep-theorem} holds. Let $T_2$ be the set of rational primes $p$ such that
\begin{enumerate}[label=\rm{(\arabic*)}]
    \item $p$ lies below a degree $1$ prime of $F$;
    \item $p$ splits into two primes in $K$ of the same inertia degree.
     \item $p$ does not lie below any prime in $S$.
\end{enumerate}
As before, $T_2$ has positive density by \Cref{one-implies-positive-density} and the assumption on the existence of a rational prime satisfying both (1) and (2). Let
\begin{align*}
    S_2:=\{v\in\Sigma_F: \N(v)=p\text{ for some }p\in T_2\}.
\end{align*}
By \Cref{p2-lemma}, if $v\in S_2$, then $v\in S^{\ord}$ or $v\in S(c)$ for some $c\in\Z$ with $0\leq c\leq 2^{2g}$. As $S(c)$ and hence $\bigcup_{0\leq c\leq 2^{2g}}S(c)$ have density $0$, it follows that $S_2$ has the same density as $S_2\cap S^{\ord}\subseteq S^{\ord}$, so it suffices to prove that $S_2$ has positive density. Now, the same argument as before gives that
\begin{align*}
    \mathrm{Density}(S_2)=\lim_{X\rightarrow\infty}\frac{\#\{v\in S_2:\N(v)\leq X\}}{\#\{v\in\Sigma_F:\N(v)\leq X\}}\geq\lim_{X\rightarrow\infty}\frac{\#\{p\in T_2:p\leq X\}}{\#\{p:p\leq X\}}=\mathrm{Density}(T_2)>0.
\end{align*}
Thus, \Cref{galois-rep-theorem} follows from the argument above.

\subsection{Density of $S(c)$}
\label{density-of-S(c)-section}
The main goal of this subsection is to show that $S(c)$ has density $0$ for each $c\in \Z$ with $0\leq c\leq 2^{2g}$, so that \Cref{galois-rep-theorem} holds by \Cref{incomplete-proof-of-galois-rep-theorem}. To prove this, we need to use the surjective property of the family of Galois representations $(\rho_\ell)_{\ell\in T}$.

Fix a rational prime $\ell$ that splits completely in $K$ and fix an embedding $\iota_\ell:\ov{\Q}\hookrightarrow\ov{\Q_\ell}$. Write $[K':\Q]=g'$ and let $\sigma'_1,\ldots,\sigma'_{g'}:K'\hookrightarrow\ov{\Q}$ be all the embeddings of $K'$ into $\ov{\Q}$. As $\ell$ splits completely in $K$ and hence in $K'$, by \Cref{places-qbarembeddings}, $\{\lambda'_i:=\lambda'_{\sigma'_i}\mid 1\leq i\leq g'\}$ is an enumeration of all the places $\lambda'|\ell$ of $K'$. In particular, we may simply make the identification $K'_{\lambda'_i}=\widehat{\iota_\ell\sigma'_i(K')}=\Q_\ell\subseteq\ov{\Q_\ell}$ so that the natural map $K'\rightarrow K'_{\lambda'_i}$ is given by $\alpha\mapsto\iota_\ell\sigma'_i(\alpha)$. In this way, 
\begin{align*}
    \cl{O}_{K'}\otimes\Z_\ell&\simarrow\prod_{i=1}^g\cl{O}_{K'_{\lambda'_i}} \\
    \alpha\otimes\beta&\mapsto(\iota_\ell\sigma'_i(\alpha)\cdot\beta)_{i=1}^{g'}.\footnotemark
\end{align*}
\footnotetext{It is easy to check that this assignment is well-defined and can be $\Z$-linearly extended to an isomorphism.}

Write $[K:K']=d$. Since $\ell$ splits completely in $K$, each $\lambda'_i$ also splits completely in $K$. In particular, for every prime $\lambda|\lambda'_i$ of $K$, we have $\mathcal{O}_{K'_{\lambda'_i}} = \mathcal{O}_{K_\lambda}$. Let us now enumerate the primes $\lambda|\ell$ of $K$ by $(\lambda_{ij})_{i=1,j=1}^{g',\;d}$ such that $\prod_{j=1}^d\lambda_{ij}=\lambda'_i$ and let $\sigma_{ij}:K\hookrightarrow\ov{\Q}$ be the embedding corresponding to the prime $\lambda_{ij}$ by \Cref{places-qbarembeddings}. We also make the identification $K_{\lambda_{ij}}=\widehat{\iota_\ell\sigma_{ij}(K)}=\Q_\ell\subseteq\ov{\Q_\ell}$ so that similarly
\begin{align*}
    \cl{O}_K\otimes\Z_\ell&\simarrow\prod_{i=1}^{g'}\prod_{j=1}^d\cl{O}_{K_{\lambda_{ij}}} \\
    \alpha\otimes\beta&\mapsto(\iota_\ell\sigma_{ij}(\alpha)\cdot\beta)_{i=1,j=1}^{g',\;d}.
\end{align*}
Thus, the composition map
\begin{align*}
    \prod_{i=1}^{g'}\cl{O}_{K'_{\lambda'_i}}\simarrow\cl{O}_{K'}\otimes\Z_\ell\hookrightarrow\cl{O}_K\otimes\Z_\ell\simarrow\prod_{i=1}^{g'}\prod_{j=1}^d\cl{O}_{K_{\lambda_{ij}}}
\end{align*}
is given by
\begin{align*}
    \prod_{i=1}^{g'}\cl{O}_{K'_{\lambda'_i}}&\hookrightarrow\prod_{i=1}^{g'}\prod_{j=1}^d\cl{O}_{K_{\lambda_{ij}}} \\
    (a_i)_{1\leq i\leq g'}&\mapsto((a_i)_{1\leq j\leq d})_{1\leq i\leq g'}.
\end{align*}
In this way, we obtain an isomorphism
\begin{align}
\label{GL2(curly-otimes-zl)-isom-copies-of-GL2(zl)}
    \GL_2(\cl{O}_K\otimes\Z_\ell)\simarrow\prod_{i=1}^{g'}\prod_{j=1}^d\GL_2(\cl{O}_{K_{\lambda_{ij}}})\simeq\prod_{i=1}^{g'}\prod_{j=1}^d\GL_2(\Z_\ell),
\end{align}
and $\GL_2(\cl{O}_{K'}\otimes\Z_\ell)$ is identified with $\prod_{i=1}^{g'}\Delta_d(\GL_2(\Z_\ell))$ under this isomorphism, where $\Delta_d:\GL_2(\Z_\ell)\rightarrow\prod_{j=1}^d\GL_2(\Z_\ell)$ is the diagonal embedding.

Now, let $\ell\in T$ (which is assumed to split completely in $K$). Let $\rho'_\ell:\Gal(\ov{F}/F)\rightarrow\prod_{i=1}^{g'}\prod_{j=1}^d\GL_2(\Z_\ell)$ be the composition of $\rho_\ell$ and the isomorphism (\ref{GL2(curly-otimes-zl)-isom-copies-of-GL2(zl)}) and let $\ov{\rho'_\ell}:\Gal(\ov{F}/F)\rightarrow\prod_{i=1}^{g'}\prod_{j=1}^d\GL_2(\F_\ell)$ be the composition of $\rho'_\ell$ and the natural reduction map $\Z_\ell\rightarrow\F_\ell$. Recall that for all primes $\ell\in T$, 
\begin{align*}
    \rho_\ell(\Gal(\ov{F}/F'))=\{M\in\GL_2(\cl{O}_{K'}\otimes\Z_\ell) : \det M\in (\Z\otimes\Z_\ell)^\times (\simeq   \Z_\ell^{\times})\}. 
\end{align*}
From this, we obtain that
\begin{align*}
    \ov{\rho'_\ell}(\Gal(\ov{F}/F'))=A_\ell(g',d):=\big\{&(M_{ij})_{i=1,j=1}^{g',\;d}\in\prod_{i=1}^{g'}\prod_{j=1}^d\GL_2(\F_\ell): M_{ij}=M_{ij'}\text{ for all }1\leq j,j'\leq d,\\
    &\hspace{0.7cm}\text{ and }\det M_{ij}=\det M_{i'j'}\text{ for all }1\leq i,i'\leq g'\text{ and }1\leq j,j'\leq d\big\}.
\end{align*}

Now, since $[F':F]$ is finite, there exists $n\in\Z^+$ such that $\tau^n\in\Gal(\ov{F}/F')$ for all $\tau\in\Gal(\ov{F}/F)$.\footnote{By considering all the cosets $\{\tau^i\Gal(\ov{F}/F')\}_{i\in\Z}$, one can deduce that $\tau^i\in\Gal(\ov{F}/F')$ for some $1\leq i\leq [F':F]$, so $n$ can be chosen as $[F':F]!$.}
In particular, this means that
\begin{align*}
    \Im\ov{\rho'_\ell}\subseteq \wt{B}_\ell(g',d,n):=&\big\{M\in \prod_{i=1}^{g'}\prod_{j=1}^d\GL_2(\F_\ell): M^n\in A_\ell(g',d)\big\} \\
    =&\big\{(M_{ij})_{i=1,j=1}^{g',\;d}\in\prod_{i=1}^{g'}\prod_{j=1}^d\GL_2(\F_\ell): M_{ij}^n=M_{ij'}^n\text{ for all }1\leq j,j'\leq d,\\
    &\hspace{1cm}\text{ and }\det M_{ij}^n=\det M_{i'j'}^n\text{ for all }1\leq i,i'\leq g'\text{ and }1\leq j,j'\leq d\big\}.
\end{align*}
In fact, $\Im\ov{\rho'_\ell}$ lies in a smaller set when $\ell>n+1$. To prove this, we first need a lemma on the matrix group $\GL_2(\F_\ell)$. 

\begin{lemma}
\label{commutator-group}
Let $n\in\Z^+$ and let $\ell>n+1$ be a prime. Let $t\in\F_\ell^\times$ such that $t^n\neq 1$. 
Then, $M\in\GL_2(\F_\ell)$ commutes with both $\sm{1&0\\0&t^n}$ and $\sm{1&\ol{n}\\0&1}$ if and only if $M=\sm{\lambda&0\\0&\lambda}$ for some $\lambda\in\F_\ell^\times$, where $\ol{n}$ denotes $n\pmod{\ell}$.
\end{lemma}
\begin{proof}
Let $M=\sm{a&b\\c&d}\in\GL_2(\F_\ell)$. Then, $M\sm{1&0\\0&t^n}=\sm{1&0\\0&t^n}M$ implies that
\begin{align*}
    \m{a&t^nb\\c&t^nd}=\m{a&b\\t^nc&t^nd}.
\end{align*}
Hence, $b(t^n-1)=c(t^n-1)=0$, implying that $b=c=0$ as $t^n\neq 1$. Now, $M\sm{1&\ol{n}\\0&1}=\sm{1&\ol{n}\\0&1}M$ implies that
\begin{align*}
    \m{a&a\ol{n}\\0&d}=\m{a&d\ol{n}\\0&d}.
\end{align*}
Hence, $a\ol{n}=d\ol{n}$, implying that $a=d$ as $\ol{n}\in\F_\ell^\times$.
\end{proof}

\begin{proposition}\label{residue-Galois-image-prop}
Let $n\in\Z^+$ be such that $\tau^n\in\Gal(\ov{F}/F')$ for all $\tau\in\Gal(\ov{F}/F)$ and let $\ell\in T$ be a prime with $\ell>n+1$. Then,
\begin{align*}
    \Im\ov{\rho'_\ell}\subseteq B_\ell(g',d,n):=&\big\{(M_{ij})_{i=1,j=1}^{g',\;d}\in\prod_{i=1}^{g'}\prod_{j=1}^d\GL_2(\F_\ell): M_{i1}^{-1}M_{ij}\in\mu_n(\F_\ell)I_2,\\
    &\hspace{0.5cm}\text{ and }\det M_{i1}^n=\det M_{i'1}^n\text{ for all }1\leq i,i'\leq d\text{ and }1\leq j\leq d\big\},
\end{align*}
where $\mu_n(\F_\ell):=\{\zeta\in\F_\ell:\zeta^n=1\}$.
\end{proposition}
\begin{remark}
Under the notation in the proposition, there is a natural identification
\begin{align*}
    B_\ell(g',1,n)\times\mu_n(\F_\ell)^{g'(d-1)}=B_\ell(g',1,n)\times\prod_{i=1}^{g'}\prod_{j=2}^d\mu_n(\F_\ell)&\simarrow B_\ell(g',d,n) \\
    \big((M_i)_{i=1}^{g'},(\zeta_{ij})_{i=1,j=2}^{g',\;d}\big)&\longmapsto    (\zeta_{ij}M_i)_{i=1,j=1}^{g',\;d},
\end{align*}
where $\zeta_{i1}:=1$ for all $1\leq i\leq g'$.
\end{remark}
\begin{proof}
Fix $\tau\in\Gal(\ov{F}/F)$ and write $\ov{\rho'_\ell}(\tau)=M=(M_{ij})_{i=1,j=1}^{g',\;d}$. 
Let $A=(A_i)_{i=1}^{g'}\in A_\ell(g',d)=\ov{\rho'_\ell}(\Gal(\ov{F}/F'))$ and let $\alpha\in\Gal(\ov{F}/F')$ be such that $\ov{\rho'_\ell}(\alpha)=A$. Then,
\begin{align*}
    MAM^{-1}=\ov{\rho'_\ell}(\tau\alpha\tau^{-1})\in\wt{B}_\ell(g',d,n).
\end{align*}
In particular,
\begin{align*}
    M_{ij}A_i^nM_{ij}^{-1}=(M_{ij}A_iM_{ij}^{-1})^n=(M_{i1}A_iM_{i1}^{-1})^n=M_{i1}A_i^nM_{i1}^{-1}
\end{align*}
or equivalently
\begin{align*}
    (M_{i1}^{-1}M_{ij})A_i^n=A_i^n(M_{i1}^{-1}M_{ij})
\end{align*}
for all $1\leq i\leq g'$ and $1\leq j\leq d$.

Since $A\in A_\ell(g',d)$ is arbitrary and $\ell>n+1$, we may pick $A$ with $A_i$ being $\sm{1&0\\0&t}$ or $\sm{1&1\\0&1}$, where $t\in\F_\ell^\times$ with $t^n\neq1$, so that $A_i^n$ is $\sm{1&0\\0&t^n}$ or $\sm{1&n\\0&1}$. In particular, it follows from \Cref{commutator-group} that
\begin{align*}
    M_{i1}^{-1}M_{ij}=\m{\lambda&0\\0&\lambda}
\end{align*}
for some $\lambda\in\F_\ell^\times$. To see that $\lambda\in\mu_n(\F_\ell)$, note that $(M_{ij})_{i=1,j=1}^{g',\;d}=\ov{\rho'_\ell}(\tau)\in\wt{B}_\ell(g',d,n)$,
so
\begin{align*}
    M_{ij}^n=M_{i'j'}^n\text{ for all }1\leq i,i'\leq g'\text{ and }1\leq j,j'\leq d.
\end{align*}
In particular,
\begin{align*}
    \m{\lambda^n&0\\0&\lambda^n}=(M_{i1}^{-1}M_{ij})^n=M_{i1}^{-n}M_{ij}^n=I_2
\end{align*}
(note that $M_{ij}$ and $M_{i1}$ commute).
Hence, $\lambda\in\mu_n(\F_\ell)$.
\end{proof}

Let $v\notin S$ be a prime of $F$ and fix a Frobenius lift $\Frob_v\in\Gal(\ov{F}/F)$ at $v$. Let $\rho'_\ell(\Frob_v)=(A_{ij})_{i=1,j=1}^{g',\;d}$. Then, $\tr A_{ij}=\iota_\ell\sigma_{ij}(a_v)$ and $\det A_{ij}=\iota_\ell\sigma_{ij}\big(\epsilon(v)\N(v)\big)$ for all $1\leq i\leq g'$ and $1\leq j\leq d$.  Note that $\epsilon$ is of finite order, so $\N_{K/\Q}(\epsilon(v))=\pm 1$, and hence $\N_{K/\Q}(\epsilon(v))^g=1$ (as $g$ is necessarily even). In this way, the condition that $\N_{K/\Q}(a_v)^2=c\N(v)^g$ is equivalent to that 
\begin{align}
\label{conj-class-1}
    \bigg(\prod_{i=1}^{g'}\prod_{j=1}^d\tr A_{ij}\bigg)^2=c\cdot\prod_{i=1}^{g'}\prod_{j=1}^d\det A_{ij}.
\end{align}
As this condition stays the same under the reduction map $\Z_\ell\rightarrow\F_\ell$, we obtain that for $v\in S(c)$,
\begin{align*}
    \ov{\rho'_\ell}(\Frob_v)\in C_\ell(g',d,n,c)\cap\Im\ov{\rho'_\ell}=\bigg\{(M_{ij})_{i=1,j=1}^{g',\;d}\in\Im\ov{\rho'_\ell}: \bigg(\prod_{i=1}^{g'}\prod_{j=1}^d\tr M_{ij}\bigg)^2=\bar{c}\cdot\prod_{i=1}^{g'}\prod_{j=1}^d\det M_{ij}\bigg\},
\end{align*}
where
$\bar{c}$ denotes $c\pmod \ell$ and 
\begin{align*}
    C_\ell(g',d,n,c):=\bigg\{&(M_{ij})_{i=1,j=1}^{g',\;d}\in B_\ell(g',d,n):\bigg(\prod_{i=1}^{g'}\prod_{j=1}^d\tr M_{ij}\bigg)^2=\bar{c}\cdot\prod_{i=1}^{g'}\prod_{j=1}^d\det M_{ij}\bigg\}.
\end{align*}
Note that $C_\ell(g',d,n,c)\cap\Im\ov{\rho'_\ell}$ is invariant under conjugation by $\Im\ov{\rho'_\ell}$. By the Chebotarev density theorem,
\begin{align}
\label{density-of-S(c)-upper-bound}
    \mathrm{Density}(S(c))\leq\frac{|C_\ell(g',d,n,c)\cap\Im\ov{\rho'_\ell}|}{|\Im\ov{\rho'_\ell}|}\leq\frac{|C_\ell(g',d,n,c)|}{|A_\ell(g',d)|}.
\end{align}
We will now prove that the latter expression goes to $0$ as $\ell\to\infty$, from which it follows that $S(c)$ has density $0$.

\begin{lemma}\label{countinglemma}
Let $\ell$ be an odd prime, $t\in \F_\ell$, and   $d\in \F_\ell^{\times}$. Then we have
    $$\#\left\{N \in \GL_2(\Z/\ell\Z) : \det N = d, \tr N =t \right\} = \ell^2 + \ell \cdot \bigg(\frac{t^2-4d}{\ell}\bigg),$$
    where $\left(\frac{\cdot}{\ell}\right)$ is the Legendre symbol. 
\end{lemma}
\begin{proof} See \cite[Lemma 2.7]{CoFoMu2005}.
\end{proof}

\begin{lemma}
\label{p2-case-asymptotics}
Let $c_0\in\Z$. Then, for all sufficiently large prime $\ell$, 
    we have 
    \[
\frac{|C_\ell(g',d,n,c_0)|}{|A_\ell(g',d)|} \ll_{g', d, n} \frac{1}{\ell}.
    \]
\end{lemma}
\begin{proof}
First, we give an asymptotic of $|A_\ell(g',d)|$.
Observe that 
\[
|A_\ell(g',d)|=|A_\ell(g')|,
\]
where for any $g\geq 1$,
\[
A_\ell(g):=\bigg\{(N_i)_{1\leq i\leq g} \in \prod_{1\leq i\leq g}\GL_2(\F_\ell):  \det N_i= \det N_{i'} \text{ for all $1\leq i, i'\leq g$} \bigg\}.
\]
It has been computed in \cite[Lemma 12]{CoWa2023} that 
$
|A_\ell(g')| \asymp \ell^{3g'+1}.
$ Hence, 
$|A_\ell(g', d)|\asymp \ell^{3g'+1}.
$

Next, we give an upper bound of $|C_\ell(g',d,n,c_0)|$.  By \Cref{residue-Galois-image-prop}, we observe that  
\begin{equation}\label{matrix-size}
|C_\ell(g',d,n,c_0)|\leq n^{g'd}\sum_{c\in  \mu_n(\F_\ell)}|D_\ell(g', d, n,c\thin\ol{c_0})|,
\end{equation}
where for any $c\in \F_\ell$, we define 
\begin{align*}
D_\ell(g', d, n,c) & :=\bigg\{(N_i)_{1\leq i\leq g'}  \in \prod_{1\leq i\leq g'}\GL_2(\F_\ell):   (\det N_i)^n= (\det N_{i'})^n \text{ for all $1\leq i, i'\leq g'$}\\
&  \hspace{3cm}  \text{and } \bigg(\prod_{1\leq i\leq g'} \tr N_i \bigg)^{2d}=c\cdot \prod_{1\leq i\leq g'} \left(\det N_i\right)^d \bigg\}.
\end{align*}
To see this, for any $(N_{i, j})_{i=1,j=1}^{g',\;d} \in C_\ell(g',d,n,c_0)$, we have $N_{i, j}=\zeta_{i, j} N_{i, 1}$ with $\zeta_{i, j}\in \mu_n(\F_\ell)$ for all $1\leq i\leq g'$ and $1\leq j'\leq d$. 
By denoting $N_i:=N_{i, 1}$, we see that there  exist $c_1, c_2\in \mu_n(\F_\ell)$ such that
\[
\bigg(\prod_{1\leq i\leq g'}\prod_{1\leq j\leq d} \tr N_{i, j} \bigg)^{2}=c_1 \bigg(\prod_{1\leq i\leq g'} \tr N_i \bigg)^{2d} \; \text{ and } \; \prod_{1\leq i\leq g'}\prod_{1\leq j\leq d} \det N_{i, j} =c_2 \bigg(\prod_{1\leq i\leq g'} \det N_i \bigg)^{d}.
\]
Finally, the factor $n^{g'd}$ in \Cref{matrix-size} arises from the possible choices of $\zeta_{i, j}$. Hence, it suffices to estimate $|D_\ell(g',d, n,c\thin\ol{c_0})|$ for $c\in \mu_n(\F_\ell)$.   
Observe that if $g'=1$, then by \Cref{countinglemma},
\begin{align*}
|D_\ell(1, d, n,c\thin\ol{c_0})| &=
\#\big\{N\in \GL_2(\F_\ell): (\tr N)^{2d}=c\thin\ol{c_0} \cdot (\det N)^d
\big\}\\
&\leq \sum_{D\in \F_\ell^{\times}} \#\{N\in \GL_2(\F_\ell): (\tr N)^{2d}=c\thin\ol{c_0} D^d, \det N=D\}\\
& \ll (\ell-1)\cdot  (2d) \ell^2\ll_d \ell^3 (=\ell^{3g'}).
\end{align*}
For $g'\geq 2$,  we will compare the size of $D_\ell(g', d, n,c\thin\ol{c_0})$ with the following matrix group
\[
A_{\ell, n}(g'):=\bigg\{(N_i)_{1\leq i\leq g'} \in \prod_{1\leq i\leq g'}\GL_2(\F_\ell):  (\det N_i)^n= (\det N_{i'})^n \text{ for all $1\leq i, i'\leq g'$} \bigg\}.
\]
Since for each $1\leq i, i'\leq g'$, $(\det N_i)^n= (\det N_{i'})^n$ implies $\det N_i\in \mu_n(\F_\ell) \det N_i'$, we obtain 
\[
|A_{\ell, n}(g')|\leq n^{g'} |A_\ell(g')|\asymp_{n, g'} \ell^{3g'+1}.
\]
Therefore, if $\ol{c_0}\neq 0$, then
\begin{align*}
|D_\ell(g', d, n,c\thin\ol{c_0})| & \leq \sum_{(N_i)_{1\leq i\leq g'-1} \in A_{\ell, n}(g'-1)}  \#\bigg\{N_{g'} \in \GL_2(\F_\ell): \det N_1^n=\dots=\det N_{g'}^n, \\
& \hspace{5cm} 
\prod_{1\leq i\leq g'-1}(\tr N_i)^{2d} \cdot (\tr N_{g'})^{2d} =
  c\thin\ol{c_0} \cdot(\det N_{g'})^d
\bigg\}\\
& =  \sum_{(N_i)_{1\leq i\leq g'-1} \in A_{\ell, n}(g'-1)}  \#\bigg\{N_{g'} \in \GL_2(\F_\ell): \det N_{g'}^n=D, \; T^{n} \cdot (\tr N_{g'})^{2dn} =
  c\thin\ol{c_0} D^{d}, \\
& \hspace{5.5cm}  
 \text{where $D:=\det N_1^n$,  $T:=\prod_{1\leq i\leq g'-1}(\tr N_i)^{2d} (\neq 0)$}
\bigg\}\\
& \ll |A_{\ell, n} (g'-1)|\cdot (2dn^2) \ell^2 \\
& \ll_{d, n} \ell^{3g'}.
\end{align*}
If $\ol{c_0}=0$, then 
we get 
\begin{align*}
|D_\ell(g', d, n, 0)| & \leq \sum_{(N_i)_{1\leq i\leq g'-1} \in D_\ell(g'-1, d, n, 0)}  \#\bigg\{N_{g'} \in \GL_2(\F_\ell): \det N_1^n=\dots=\det N_{g'}^n\bigg\} \\
& \hspace{1cm} + \sum_{(N_i)_{1\leq i\leq g'-1} \in A_{\ell, n}(g'-1)}  \#\bigg\{N_{g'} \in \GL_2(\F_\ell): \det N_{g'}^n=D, \; T \cdot (\tr N_{g'})^{2d} =0,\\
& \hspace{5.6cm} \text{where $D:=\det N_1^n, \; T:=\prod_{1\leq i\leq g'-1}(\tr N_i)^{2d}(\neq 0)$} \bigg\}\\
& \ll |D_\ell(g'-1, d, n, 0)|\cdot n\ell^3 + |A_{\ell, n}(g'-1)|\cdot n\ell^2 \\
& \ll_{d, n} \ell^{3g'},
\end{align*}
where in the last line we use the induction hypothesis $|D_\ell(g'-1, d, n, 0)| \ll \ell^{3(g'-1)}$ on $g'-1$.
In all,  for all sufficiently large $\ell$, we have
\[
\frac{|C_\ell(g', d, n,cc_0)|}{|A_{\ell}(g', d)|}\ll_{g', d, n} \frac{\ell^{3g'}}{\ell^{3g'+1}} =\frac{1}{\ell}.
\]
\end{proof}

\begin{corollary}
\label{density-of-S(c)}
Under the assumptions on the family of Galois representations in \Cref{galois-rep-theorem}, the set $S(c)$ has density $0$.
\end{corollary}
\begin{proof}
The proof  follows easily from \Cref{density-of-S(c)-upper-bound} and \Cref{p2-case-asymptotics}.
\end{proof}

\section{Applications of \Cref{galois-rep-theorem} and \Cref{inert-in-K-theorem}}\label{sec:application}

\subsection{Applications to abelian varieties}
\label{sec:application-abelian-variety}

We now provide the proof of \Cref{ablian-var-theorem-2}, assuming \Cref{modular-form-theorem} (whose proof will be given in \Cref{sec:application-modular-form}),  the proof of \Cref{ablian-var-theorem-3} using \Cref{galois-rep-theorem} and \Cref{inert-in-K-theorem}, and the proof of \Cref{abelian-var-quadratic-theorem} using \Cref{ablian-var-theorem-3}. We adopt all the notations from \Cref{ablian-var-theorem-2},  \Cref{abelian-var-quadratic-theorem}, and \Cref{ablian-var-theorem-3}, respectively. 

\begin{proof}[Proof of \Cref{ablian-var-theorem-2}]
By the assumptions, $A$ is a $\GL_2$-type abelian variety. By the proof of \cite[Theorem~2.1]{Ri2004}, 
$A$ is isogenous over $\Q$ to $B^r$ for some simple \mbox{$\GL_2$-type} abelian variety $B$ over $\Q$. Moreover, by \cite[Theorem~2.1]{Ri2004} and its proof, $K':=\End_\Q(B)\otimes\Q$ is a number field of degree equal to $\dim B$ with $K'\hookrightarrow K\hookrightarrow\End_\Q(A)\otimes\Q$. Since $[K:\Q]=q$ (resp.~$2q$), $[K':\Q]=1$ or $q$ (resp.~$1$, $2$, $q$, or $2q$). If $[K':\Q]=1$, then $B/\Q$ is an elliptic curve. In particular, the set of ordinary primes of $B$ has positive density, and so does the set of ordinary primes of $A$.

Suppose that $[K':\Q]\neq 1$. By \cite[Theorem~4.4]{Ri2004}, under Serre's modularity conjecture, $B$ is isogenous over $\Q$ to abelian varieties of the form $A_f$ associated to some newform $f$ of weight $2$ under the Eichler--Shimura theory. As Serre's modularity conjecture is proved in \cite{KhWi2009-I,KhWi2009-II}, we can thus complete the proof by invoking \Cref{modular-form-theorem}, \Cref{f-ordinary-iff-Af-ordinary}, and the isomorphism $K'=\End_{\Q}(A_f)\otimes\Q\simeq K_f$.
\end{proof}

\begin{proof}[Proof of \Cref{ablian-var-theorem-3}]

Recall that $A$ is absolutely simple, $\End_F(A)\otimes\Q$ contains a number field $K$ of degree $g$, and $A/\ol{F}$ is of the first kind. From these, it is easy to check that $A$ is of $\GL_2(K)$-type and satisfies the assumptions in \Cref{degree-2-char-poly} and \Cref{Galois-images}. In particular, by the discussions in 
 \Cref{section-rep-abelian}, the family of $\ell$-adic Galois representations $\{\rho_{A,\ell}:\Gal(\ol{F}/F)\rightarrow\GL_2(K\otimes\Q_\ell)\}_{\ell\in\Sigma_\Q}$ satisfies that
\begin{enumerate}
    \item each $\rho_{A,\ell}$ is unramified outside $\{v \in \Sigma_F: v|N_A \ell\}$;
    \item for every prime $v\nmid N_A\ell$ of $F$, $\det(XI-\rho_{A,\ell}(\Frob_v))=X^2-(a_v(A)\otimes 1)X+(\epsilon(v)\N(v)\otimes1)$, which is independent of $\ell$ (where $\epsilon:\Gal(\ol{F}/F)\rightarrow K^\times$ is some finite order character independent of $\ell$);
    \item $|\iota(a_v(A))|_\C\leq 2\sqrt{\N(v)}$ for all primes $v\nmid N_A$ of $F$ and all embeddings $\iota:K \hookrightarrow\C$;
    \item there exist two finite extensions of fields $F'/F$ and $K/K'$ such that  for $\ell$ sufficiently large, we have, up to conjugation, that
\[
\rho_{A, \ell}(\Gal(\ol{F'}/F'))=
\{M\in\GL_2(\cl{O}_{K'}\otimes\Z_\ell) : \det M\in \Z_\ell^{\times}\}.
\]
\end{enumerate}

Therefore, if we are in the case of (P2), then the result readily follows from \Cref{galois-rep-theorem} (with $T$ being the set of rational primes that split completely in $K$) and \Cref{ordinary-lambda-trace}. In the case of (P1), we need to show that the set $\{v\in\Sigma_F:v\nmid N_A\text{ and }a_v(A)=0\}$ has density~$0$, which simply follows from \Cref{density-of-S(c)} when $c=0$. The case of (P1) then follows from \Cref{inert-in-K-theorem} and \Cref{ordinary-lambda-trace}.
\end{proof}

\begin{proof}[Proof of \Cref{abelian-var-quadratic-theorem}]
By the assumptions, $A$ is a $\GL_2$-type abelian variety. Since the CM case is covered by e.g.~\cite[Remark 12]{Fi2024} and the case when $q=2$ is covered by \cite{Sawin2016}, we will thus assume that $A$ is not a CM abelian variety and $A$ is of dimension $q\geq 3$.

Let $F'/F$ be a finite extension such that $\End_{F'}(A)\otimes\Q=\End_{\ol{F}}(A)\otimes\Q$. Since $\dim A=q$ is a prime and $A$ is not a CM abelian variety, if $A/F'$ is not simple, then by \cite[pp. 908--909 and Proposition 1.5]{Wu2020}, it must be isogenous over $F'$ to either $B^q$ for some non-CM elliptic curve $B$ over~$F'$ or $B_1^{r_1}\times B_2^{r_2}$ for some CM abelian varieties $B_1,B_2$ over $F'$ with $r_1\dim B_1=r_2\dim B_2=q/2$. As $q$~is an odd prime, the latter case is not possible. In the former case, the set of ordinary primes of $B$ over $F'$ has density~$1$, so the set of ordinary primes of $A$ over $F$ has density at least $\frac{1}{[F':F]}$.

Now, suppose that $A/F'$ is simple. Note that $q$ is an odd prime, $A$ is not a CM abelian variety, and $K\hookrightarrow\End_{F'}(A)\otimes\Q$. By \cite[Proposition~1.7]{Wu2020}, we must have that $K=\End_{F'}(A)\otimes\Q$ is a totally real field. In particular, 
\begin{align*}
    K=\End_{F'}(A)\otimes\Q=\End_{\ol{F}}(A)\otimes\Q.
\end{align*}
This implies that $A/\ol{F}$ must be simple as $\End_{\ol{F}}(A)\otimes\Q$ is a field. In this way, $A$ is absolutely simple, $\End_F(A)\otimes\Q$ contains a number field $K$ of degree $q=\dim A$, and $A/\ol{F}$ is not of type IV (in fact,  $A$ is of type I), so $A$ satisfies the assumptions in \Cref{ablian-var-theorem-3}. As $F$ is a quadratic field and $K$ is a number field whose degree is an odd prime, it follows from \Cref{quadratic-prime-splitting} that the condition (P1) in \Cref{ablian-var-theorem-3} is satisfied. Hence, the set of ordinary primes of $A$ over $F$ has positive density by \Cref{ablian-var-theorem-3}.
\end{proof}

\subsection{Application to modular forms}
\label{sec:application-modular-form}

We will now prove \Cref{modular-form-theorem} and we adopt all the notations from \Cref{modular-form-theorem} and \Cref{galois-rep-of-modular-form-section}.

First, note that if $f$ is CM, then the set of ordinary primes of $f$ has positive density by \Cref{ordinary-primes-of-cm-modular-forms}. We will thus assume that $f$ is non-CM from now on. 
By the discussions in 
 \Cref{galois-rep-of-modular-form-section}, the family of $\ell$-adic Galois representations $\{\rho_{f,\ell}:\Gal(\ol{\Q}/\Q)\rightarrow\GL_2(K_f\otimes\Q_\ell)\}_{\ell\in\Sigma_\Q}$ satisfies that
\begin{enumerate}[label=(\arabic*)]
    \item each $\rho_{f,\ell}$ is unramified at rational primes $p$ with $p\nmid N\ell$;
    \item for every rational prime $p\nmid N\ell$, $\det(XI-\rho_{f,\ell}(\Frob_p))=X^2-(a_p(f)\otimes 1)X+(\epsilon(p)p\otimes1)$, which is independent of $\ell$;
    \item $|\iota(a_p(f))|_\C\leq 2\sqrt{p}$ for all $p\nmid N$ and all embeddings $\iota:K_f\hookrightarrow\C$;
    \item the set $\{p\nmid N:a_p(f)=0\}$ of rational primes has density $0$;
    \item there exist two finite extensions of fields $F'_f/\Q$ and $K_f/K'_f$ such that for $\ell$ sufficiently large, we have, up to conjugation, that
    \begin{align*}
        \rho_{f,\ell}(\Gal(\ov{\Q}/F'_f))=\{M\in\GL_2(\cl{O}_{K'_f}\otimes\Z_\ell) : \det M\in\Z_\ell^{\times}\}.
    \end{align*}
\end{enumerate}
Here the only property that is not recalled in \Cref{galois-rep-of-modular-form-section} is (4), which follows from \cite[Section~7.2]{Se1981}.

From these properties, we can already deduce the second part of \Cref{modular-form-theorem} by applying \Cref{galois-rep-theorem} (with $T$ being the set of rational primes that split completely in $K_f$) or \Cref{inert-in-K-theorem}. 
The first part of \Cref{modular-form-theorem}, i.e., when $[K_f:\Q]=q$ or $2q$ for some prime $q$, now follows from \Cref{field-of-degree-q-or-2q} and the second part.
\\

\subsection*{Acknowledgments} 
We wish to thank Francesc Fit\'{e} for answering our questions and providing  helpful comments on Galois representations of $\mathrm{GL}_2$-type abelian varieties. We would like to thank  Rachel Pries for the stimulating conversations. We also wish to thank the Max Planck Institute for Mathematics in Bonn for its support and inspiring atmosphere for collaborations. Lastly, we are grateful for the many helpful comments and suggestions the referees provided.

\subsection*{Data availability}
There is no data used in this article.

\subsection*{Conflict of interest}
The authors state that there is no conflict of interest.

\begin{appendix}

\section{Ordinary primes of CM modular forms of weight $2$}

In this appendix, we will prove that CM modular forms of weight $2$ have positive density ordinary primes. In fact, we will prove a result purely for certain Galois representations, from which the result for CM modular forms of weight $2$ follows. The result should already be known to experts, and the approach we take here is essentially the same as \cite[Remark~12]{Fi2024}.

\begin{proposition}
\label{potent-split-implies-pos-dens-ordinary}
Let $F$ and $K$ be two number fields, let $\lambda$ be a fixed prime of $K$, and let $\epsilon$ be a Hecke character of $F$ of finite order with values in $K^\times$. Let $\rho:\Gal(\ov{F}/F)\rightarrow\GL_2(K_\lambda)$ be a continuous Galois representation such that for all primes $v$ of $F$ outside a finite set $S$ of primes of~$F$, $\rho$ is unramified at $v$ and
\begin{align*}
    \det(XI-\rho(\Frob_v))=X^2-\tr\rho(\Frob_v)X+\epsilon(v)\N(v)\in\cl{O}_K[X].
\end{align*}
Suppose that there exists a finite extension $F'/F$ such that for all but finitely many primes $v'$ of $F'$,  $\rho|_{\Gal({\ov{F}/F'})}$ is unramified at $v'$ and $\det(XI-\rho|_{\Gal({\ov{F}/F'})}(\Frob_{v'}))$ splits into linear factors over $K$. Then,
\begin{align*}
    S^{\ord}:=\{v\in\Sigma_F:v\notin S\text{ and }\mu\nmid(\tr\rho(\Frob_v))\text{ for all }\mu\in\Sigma_K\text{ with }\mu|\N(v)\}
\end{align*}
has density greater than or equal to $\frac{1}{[F':F]}$.
\end{proposition}
\begin{proof}
We first extend the finite set $S$ to include primes of $F$ that ramify in $F'/F$ or lie above rational primes that ramify in $K/\Q$.  Now, let $S'$ be a finite set of primes of $F'$ that contains all the ramified primes of $\rho':=\rho|_{\Gal(\ov{F}/F')}$ and all the primes that lie above primes in $S$.

Let $v'\notin S'$ be a prime of $F'$ of degree $1$, let $v$ be the prime of $F$ lying below $v'$, and let $p=\N(v')=\N(v)$ be the rational prime lying below $v'$. As $v'$ is of degree~$1$, the residue fields of $v'$ and $v$ coincide, so there exists a Frobenius lift at $v'$ in $\Gal(\ov{F}/F')$ such that it is also a Frobenius lift at $v$ in $\Gal(\ov{F}/F)$.
Hence,
\begin{align*}
    \det(XI-\rho(\Frob_{v'}))=\det(XI-\rho(\Frob_{v}))=X^2-a_vX+\epsilon(v)p=(X-\alpha_v)(X-\beta_v)
\end{align*}
for $a_v\in K$ and $\alpha_v,\beta_v\in K$ (by the assumptions). Suppose that $v\notin S^{\ord}$. Then, there exists a prime $\mu$ of $K$ with $\mu|p$ such that $\mu|(a_v)$. 
We thus have
\begin{align*}
    v_{\mu}(a_v)\geq 1\quad\text{ and }\quad v_{\mu}(\epsilon(v)p)\geq 1.
\end{align*}
These imply that $v_{\mu}(\alpha_v)\geq 1$ and $v_{\mu}(\beta_v)\geq 1$. As $\epsilon(v)$ is a unit,
\begin{align*}
    v_{\mu}(p)=v_{\mu}(\epsilon(v)p)=v_{\mu}(\alpha_v)+v_{\mu}(\beta_v)\geq 2,
\end{align*}
so $p$ ramifies in $K/\Q$. This contradicts that $v'\notin S'$ by the definition of $S'$. Thus, $v\in S^{\ord}$.

To summarize, if a prime $v$ of $F$ lies below a degree $1$ prime $v'\notin S'$ of $F'$, then $v\in S^{\ord}$, i.e.,
\begin{align*}
    S_0:=\{v\in\Sigma_F:v'|v\text{ for some degree $1$ prime $v'\notin S'$ of $F'$}\}\subseteq S^{\ord},
\end{align*}
so it suffices to show that $S_0$ has density greater than or equal to $\frac{1}{[F':F]}$. Note that each $v\in S_0$ may split into at most $[F':F]$ primes of $F'$ of degree $1$, so
\begin{align*}
    [F':F]\cdot\#\{v\in S_0:\N(v)\leq X\}\geq\#\{\text{degree $1$ prime $v'$ of $F'$}:\N(v')\leq X\}.
\end{align*}
On the other hand, by Landau's prime ideal theorem,
\begin{align*}
    \#\{v\in\Sigma_F:\N(v)\leq X\}\sim\frac{X}{\log X}\sim\#\{v'\in\Sigma_{F'}:\N(v')\leq X\}.
\end{align*}
Hence,
\begin{align*}
    \mathrm{Density}(S_0)=\lim_{X\to\infty}\frac{\#\{v\in S_0:\N(v)\leq X\}}{\#\{v\in\Sigma_F:\N(v)\leq X\}}&\geq\frac{1}{[F':F]}\lim_{X\to\infty}\frac{\#\{\text{degree $1$ prime $v'$ of $F'$}:\N(v')\leq X\}}{\#\{v\in\Sigma_F:\N(v)\leq X\}} \\
    &=\frac{1}{[F':F]}\lim_{X\to\infty}\frac{\#\{\text{degree $1$ prime $v'$ of $F'$}:\N(v')\leq X\}}{\#\{v'\in\Sigma_{F'}:\N(v')\leq X\}} \\
    &=\frac{1}{[F':F]}.
\end{align*}
\end{proof}

We are now ready to prove the result for CM modular forms of weight $2$. We will take all the facts on CM modular forms for granted and refer readers to \cite[Section~3~and~4]{Ri1977} for references.

\begin{corollary}
\label{ordinary-primes-of-cm-modular-forms}
The set of ordinary primes of any CM newform of weight $2$ has density $1/2$.
\end{corollary}
\begin{proof}
Let $f$ be a CM newform of weight $2$. Let $K_f$ be the coefficient field of $f$, let $\lambda$ be a prime of~$K_f$, and let $\rho_{f,\lambda}:\Gal(\ov{\Q}/\Q)\rightarrow\GL_2(K_{f,\lambda})$ be the $\lambda$-adic representation attached to $f$.

First, since there exists an imaginary quadratic field $F'$ such that $a_p(f)=0$ for all but finite many primes $p$ that are inert in $F'$, it follows that the set of ordinary primes of $f$ has density at most~$1/2$. For the lower bound, note that there exists an algebraic Hecke character $\psi$ of~$F'$ such that $\rho_{f,\lambda}\simeq\mathrm{Ind}_{F'/\Q}\psi$ (after extending scalars). In particular, there exists a number field $K$ (where the image of $\psi$ lies) such that for all but finite many primes $v'$ of $F'$, $\rho':=\rho_{f,\lambda}|_{\Gal(\ov{\Q}/F')}$ is unramified at $v'$ and $\det(XI-\rho'(\Frob_{v'}))$ splits into linear factors
over $K$. As $f$ is of weight $2$, the data $(\rho_{f,\lambda},F'/\Q,K)$ satisfies the condition in \Cref{potent-split-implies-pos-dens-ordinary}, so it follows that the set of ordinary primes of $f$ has density at least $1/2$. The result then follows.
\end{proof}

\end{appendix}

\bibliographystyle{amsalpha}
\bibliography{Ref}

\end{document}